\documentclass[12pt]{amsart}

\usepackage{dsfont}
\usepackage{tikz} \usepackage{tikz-cd}
\usepackage{amssymb}
\usepackage{amsmath}
\usepackage{kbordermatrix}
\usepackage{graphicx}

\usepackage{marvosym}

\usetikzlibrary{shapes,fit,intersections}
\usetikzlibrary{decorations.markings}
\usetikzlibrary{decorations.pathreplacing}
\usetikzlibrary{patterns}
\usetikzlibrary{matrix,arrows}
\usepgflibrary{decorations.pathmorphing}

\DeclareMathOperator\Loc{Loc}
\DeclareMathOperator\id{id}
\DeclareMathOperator\Sym{Sym}
\DeclareMathOperator\Gr{Gr}
\DeclareMathOperator\PPP{\mathbb P}
\DeclareMathOperator\h{\hbar}
\DeclareMathOperator\End{End}
\DeclareMathOperator\I{\mathcal I}

\DeclareMathOperator\cR{\mathcal R}
\DeclareMathOperator\C{\mathbb C}
\DeclareMathOperator\T{\mathbb T}
\DeclareMathOperator\X{\mathcal X}
\DeclareMathOperator\Hom{\mathrm Hom}
\DeclareMathOperator\pt{pt}
\DeclareMathOperator\Stab{Stab}
\DeclareMathOperator\Y{\mathcal Y}
\DeclareMathOperator\Grkn{\Gr_k\!\C^n}
\DeclareMathOperator\St{L}
\DeclareMathOperator\W{\mathbb W}
\DeclareMathOperator\Hb{\mathbb H}
\DeclareMathOperator{\Tr}{Tr}
\DeclareMathOperator{\Crit}{Crit}
\DeclareMathOperator\even{even}
\DeclareMathOperator\odd{odd}

\def\zz{\scalebox{.6}{(00)}}
\def\zo{\scalebox{.6}{(01)}}
\def\oz{\scalebox{.6}{(10)}}
\def\oo{\scalebox{.6}{(11)}}

\newcommand{\vin}{\rotatebox[origin=c]{90}{$\in$}}
\def\Xrn{\X^{(r)}_n}

\makeatletter
\renewcommand*\env@matrix[1][*\c@MaxMatrixCols c]{%
  \hskip -\arraycolsep
  \let\@ifnextchar\new@ifnextchar
  \array{#1}}
\makeatother

\newtheorem{fact}{Fact}[section]
\newtheorem{lemma}[fact]{Lemma}
\newtheorem{theorem}[fact]{Theorem}
\newtheorem{definition}[fact]{Definition}
\newtheorem{example}[fact]{Example}
\newtheorem{rremark}[fact]{Remark}
\newenvironment{remark}{\begin{rremark} \rm}{\end{rremark}}
\newtheorem{proposition}[fact]{Proposition}

\def\aedg{NS-edge}
\def\bedg{D-edge}
\def\avr{NS-variety}
\def\bvr{D-variety}
\def\GL{ \mathrm{GL}}
\def\GLq{ \GL_q }
\def\GLqv#1{ \GLq(#1) }
\def\GLqn{ \GLqv{n}}
\def\GLv#1{ \GL(#1) }
\def\gl{ \mathrm{gl}}
\def\gl{\mathfrak {gl}}
\def\glv#1{ \gl(#1) }
\def\gln{ \glv{n}}
\def\glnmax{\glv{\nmax}}
\def\glnv{ \glv{\nv}}
\def\glsv#1#2{ \gl(#1|#2) }
\def\glmn{ \glsv{\gM}{\rN}}

\def\glzn{ \glsv{0}{\gN}}
\def\glnz{ \glsv{\gN}{0}}

\def\glN{ \glv{\gN}}

\def\xvo{ \xv_1 }
\def\xvt{ \xv_2 }
\def\nvo{ n_{\xvo}}
\def\nvt{ n_{\xvt}}
\def\GLn{ \GLv{n}}
\def\GLm{ \GLv{m}}
\def\GLN{ \GLv{\gN}}
\def\GLnmax{\GLv{\nmax}}
\def\gN{ N }
\def\GLN{ \GLv{\gN}}

\def\edg{ e }
\def\vrt{ v }
\def\rmst{\mathrm{st}}
\def\xX{ \mathcal{X}}

\def\xsm{\mathrm{s}}
\def\xWsm{ \xW^{\xsm}}
\def\xXsm{ \xX^{\xsm}}
\def\xXsmv#1{ \xXsm_{#1}}

\def\xXsmi{ \xXsmv{i}}

\def\xXsme{ \xXsmv{\edg}}
\def\xXsmi{ \xXsmv{i}}
\def\xXe{ \xX_{\edg}}
\def\xWe{ \xW_{\edg}}
\def\xWsme{ \xWsm_{\edg}}
\def\xmu{ \mu }

\def\xmui{ \xmu_i}
\def\xmuv{ \xmu_\vrt }
\def\rme{ \mathrm{e}}
\def\rmv{ \mathrm{v}}
\def\nv{ n_\vrt}
\def\GLnv{ \GLv{\nv}}

\def\GLVi{ \GLv{\xVi}}

\def\rmNS{ \mathrm{NS}}
\def\rmD{ \mathrm{D}}

\def\rmT{ \mathrm{T}}
\def\rmTs{ \rmT^* }

\def\cL{ \mathcal{L} }
\def\cLv#1#2{ \cL_{(#1,#2)}}
\def\cLXW {\cLv{\xX}{W}}

\def\tcL{ \widetilde{\cL} }
\def\tcLv#1#2{ \tcL_{(#1,#2)}}
\def\tcLXW {\tcLv{\xX}{W}}
\def\tcLXWv#1{\tcLv{\xX_{#1}}{\xW_{#1}}}

\def\tcLvrv#1{ \tcL_{#1}}
\def\tcLvr{ \tcLvrv{v}}

\def\hmrd{ /\!\!/ }

\def\xQ{ Q }
\def\xXQ{ \xX_\xQ }

\def\Gv#1{ G_{#1}}

\def\Gi{ \Gv{i}}

\def\Gimo{ \Gv{i-1}}

\def\xom{\omega}

\def\muv#1{ \xmu_{#1}}

\def\muvoe{ \muv{\xv_1,\xe}}
\def\muvte{ \muv{\xv_2,\xe}}

\def\rleft{ \mathrm{L}}
\def\rright{ \mathrm{R}}
\def\mulv#1{ \xmu^{\rleft}_{#1}}
\def\murv#1{ \xmu^{\rright}_{#1}}
\def\xmunv#1{ \xmu_{#1}}
\def\xmunn{ \xmunv{n}}
\def\xmunm{ \xmunv{m}}

\def\Xvv#1{ X_{#1}}
\def\Xvoe{ \Xvv{\xv_1}}
\def\Xvte{ \Xvv{\xv_2}}

\def\xW{ W }

\def\xv{ v }
\def\xe{ e }
\def\xn{ n }

\def\xnv#1{ \xn_{#1}}
\def\xni{ \xnv{i}}
\def\xnimo{ \xnv{i-1}}
\def\xnio{ \xnv{i+1}}
\def\xnz{ \xnv{0}}

\def\bfv{ \mathbf{\xv}}

\def\xXQ{ \xX_Q }

\def\bfe{ \mathbf{e}}
\def\hXQ{ \xX_{\bfe}}
\def\GQ{ G_Q }
\def\GQ{ G_{\mathbf{v}}}

\def\Critvv#1#2{ \Crit\bigl(#1;#2\bigr) }

\def\xxG{ G }
\def\xxg{ \mathfrak{g}}

\def\ism{ \xsm }
\def\ilg{ \mathrm{lg}}
\def\smcap{ \overset{\ism}{\cap} }
\def\lgcap{ \overset{\ilg}{\cap}}
\def\gQ{ \xxg_{\bfv}}

\def\xWsmet{ \xWsm_{\bfe}}

\def\Gvpr#1{$#1$-pair}
\def\Gpr{\Gvpr{\xxG}}

\def\xGprv#1#2{ (#1;#2) }
\def\xGprbv#1#2{ \bigl(#1;#2\bigr) }

\def\rmLg{\mathrm{LG}}
\def\fLg{ f_{\rmLg}}
\def\Lgv#1{ {#1}^{\rmLg} }
\def\Lgvv#1#2{ {#1}^{\rmLg,#2}}

\def\aedg{arrow edge}
\def\bedg{bow edge}

\def\avr{arrow variety}
\def\avrs{arrow varieties}
\def\bvr{bow variety}
\def\bvrs{bow varieties}

\def\xV{ V }

\def\xVnv#1{ \xV_{#1}}
\def\xVi{ \xVnv{i}}

\def\Homst{ \Hom^{\rmst}}

\def\xcn{ n }
\def\xcnLv#1{ \xcn_{\mathrm{L},#1}}
\def\xcnRv#1{ \xcn_{\mathrm{R},#1}}
\def\xcnv#1{ \xcn_{#1}}
\def\xcne{ \xcnv{\edg}}

\def\nvv#1{ n_{\xv_{#1}}}
\def\nvo{ \nvv{1}}
\def\nvt{ \nvv{2}}

\def\fesl{\mathcal{B}}
\def\feslv#1#2{ \fesl_{#2,#1}}
\def\feslnm{ \feslv{n}{m}}

\def\varr{\mathcal{A}}
\def\varrv#1#2{ \varr_{#2,#1}}
\def\varrnm{ \varrv{n}{m}}
\def\varrst{\varr^{\rmst}}
\def\varrstv#1#2{ \varrst_{#2,#1}}
\def\varrstnm{ \varrstv{n}{m}}

\def\Unm{U_{n,m}}
\def\Unm{ U_{m,n}}

\def\xw{ w }
\def\xwp{ \xw'}
\def\bfw{ \mathbf{\xw}}
\def\bfwp{ \bfw' }

\def\rmsym{ \mathrm{sym}}

\def\xR{ R }
\def\xRs{ \xR^{\rmsym}}
\def\xRv#1{ \xR_{#1}}
\def\xRsv#1{ \xRs_{#1}}
\def\xk{ k }
\def\xRk{ \xRv{\xk}}
\def\xRo{ \xRv{1}}
\def\xRbk{ \xRv{\bfhw} }
\def\xRsbk{ \xRsv{\bfhw}}
\def\xRbkz{ \xRv{\bfhw;\bfzp}}
\def\xRbkp{ \xRv{\bfhw;\bfhwp}}
\def\xRbo{ \xRv{\bfo}}

\def\xRfl{ \xR' }
\def\xRflv#1{ \xRfl_{#1}}
\def\xRflk{ \xRflv{\xk}}
\def\xRflo{\xRflv{1}}

\def\xWd{ \Lambda }

\def\hw{ k }
\def\hwp{ \hw' }
\def\bfhw{ \mathbf{\hw}}
\def\bfhwp{ \bfhw' }

\def\rM{ K }
\def\rMp{ \rM' }
\def\gM{ M }
\def\rN{ N }

\def\wV{ V }
\def\wVv#1#2{ \wV_{#1}^{#2}}
\def\wVbwk{ \wVv{\bfhw}{\bfw}}
\def\wVbwkp{ \wVv{\bfhw;\bfhwp}{\bfw;\bfwp}}

\def\xlq{linear quiver}
\def\sprd{separated}

\def\xQvv#1#2{ Q_{#1}^{#2}}
\def\xQbwk{ \xQvv{\bfhw}{\bfw}}
\def\xQbwkp{ \xQvv{\bfhw;\bfhwp}{\bfw;\bfwp}}
\def\xQbwpk{ \xQvv{\bfhw}{\bfw;\bfwp}}

\def\rarr{ \mathrm{arr}}
\def\rbow{\mathrm{bow}}
\def\Qa{ Q^{\rarr}}
\def\Qb{ Q^{\rbow}}

\def\xQav#1{ \Qa_{#1}}
\def\xQbv#1{ \Qb_{#1}}

\def\xQaw{ \xQav{\bfw}}
\def\xQawp{ \xQav{\bfw,\bfwp}}
\def\xQbw{ \xQbv{\bfhw}}

\def\xXvv#1#2{ \xX_{#1}^{#2}}
\def\xXbwk{ \xXvv{\bfhw}{\bfw}}
\def\xXbko{ \xXvv{\bfo}{\bfw}}

\def\xXbwpk{ \xXvv{\bfhw}{\bfw;\bfwp}}
\def\xXbwpko{ \xXvv{\bfo}{\bfw;\bfwp}}

\def\xYvv#1#2{ \mathcal{Y}_{#1}^{#2}}
\def\xYbwk{ \xYvv{\bfhw}{\bfw}}

\def\xFL{ \mathcal{F}}
\def\xFLv#1{ \xFL_{#1}}
\def\xFLbw{ \xFLv{\bfw}}

\def\xSSo{ \mathcal{S}}
\def\xSS{\widehat{\xSSo}}
\def\xSSv#1{ \xSS_{#1}}
\def\xSSbhw{ \xSSv{\bfhw}}
\def\xSSbo{ \xSSv{\bfo}}

\def\xSSov#1{ \xSSo_{#1}}
\def\xSSobhw{ \xSSov{\bfhw}}

\def\Ad{ \mathrm{Ad}}
\def\Adv#1{ \Ad_{#1}}

\def\xmrk{marked}

\def\NQ{ N_Q }

\def\xmv#1{ (#1) }
\def\xmimo{ \xmv{i-1}}
\def\xmi{ \xmv{i} }
\def\xmio{ \xmv{i+1} }
\def\xmz{ \xmv{0} }
\def\xmN{ \xmv{\NQ}}

\def\flF{ \mathcal{F}}
\def\flFv#1{ \flF_{#1}}
\def\flFw{ \flFv{\bfw}}
\def\flFwm{ \flFv{\bfwp,\bfw}}

\def\hflF{ \widehat{\flF} }
\def\hflFm{ \hflF_{\bfwp,\bfw} }
\def\Fb{ F_\bullet }

\def\xsV{ V }
\def\xsVFb{ \xsV(\Fb) }
\def\xphi{ \phi }

\def\nmax{ n_{\mathrm{max}}}
\def\sint{symplectic intersection}

\def\xab#1{ |#1| }

\def\bfo{ \mathbf{1}}
\def\bfz{ \mathbf{0}}
\def\bfzp{ \bfz' }

\def\xmuS{ \xmu_{\xSS}}
\def\xmuF{ \xmu_{\flF} }

\def\xRN{ R_{(\gN)}}

\def\smS{ \mathrm{S}}
\def\smSv#1{ \smS^{#1}}

\def\Pr{ \mathcal{P}}
\def\Prv#1{ \Pr_{#1}}
\def\Prbw{\Prv{\bfw}}

\def\Cm{ \C^{\ynt}}
\def\Cs{ \C^\times }
\def\Cse{ (\Cs)_\edg }
\def\Csbow{ (\Cs)_{\mathrm{bow}}}

\def\bfz{ \mathbf{z}}
\def\Mbz{ M_{\bfz}}

\def\rmnil{ \mathrm{nil}}
\def\xwnil{ w^{\rmnil}}
\def\bfwnil{ \bfw^{\rmnil}}
\def\gMnil{ \gM^{\rmnil}}
\def\gMnil{ m }

\def\IR{\mathbb{R}}
\def\IRv#1{ \IR^{#1}}
\def\IRte{ \IRv{10}}
\def\IRth{ \IRv{3}}
\def\IRtw{ \IRv{2}}
\def\IRo{ \IRv{1}}
\def\rmquiv{\mathrm{quiv}}
\def\rmcmn{ \mathrm{cmn}}
\def\rmNS{ \mathrm{NS}5}
\def\rmD{ \mathrm{D}5}
\def\IRq{ \IRo_{\rmquiv}}
\def\IRcmn{ \IRth_{\rmcmn}}
\def\IRtwcmn{ \IRtw_{\rmcmn}}
\def\IRthNS{ \IRth_{\rmNS}}
\def\IRtwNS{\IRtw_{\rmNS}}
\def\IRthD{ \IRth_{\rmD}}
\def\IRoD{ \IRo_{\rmDfr}}
\def\IRtwx{ \IRtw_{\rmx} }
\def\IRtwy{ \IRtw_{\rmy}}

\def\rmx{ \mathrm{x}}
\def\rmy{ \mathrm{y}}

\def\rmDfr{ \mathrm{D4}}
\def\rmNS{ \mathrm{NS5}}
\def\Dx{$\rmDfr_\rmx$}
\def\Dy{$\rmDfr_\rmy$}
\def\NSx{$\rmNS_\rmx$}
\def\NSy{$\rmNS_\rmy$}

\title{New quiver-like varieties and Lie superalgebras }

\author{R. Rim\'anyi}
\address{Department of Mathematics, University of North Carolina at Chapel Hill, USA}
\email{rimanyi@email.unc.edu}

\author{L. Rozansky}
\address{Department of Mathematics, University of North Carolina at Chapel Hill, USA}
\email{rozansky@email.unc.edu}

\begin{document}

\begin{abstract} In order to extend the geometrization of Yangian $R$-matrices from Lie algebras $\gln$ to superalgebras $\gl(M|N)$, we introduce new quiver-related varieties which are associated with representations of $\glsv{M}{N}$. In order to define them similarly to the Nakajima-Cherkis varieties, we reformulate the construction of the latter by replacing the Hamiltonian reduction with the intersection of generalized Lagrangian subvarieties in the cotangent bundles of Lie algebras sitting at the vertices of the quiver. The new varieties come from replacing some Lagrangian subvarieties with their Legendre transforms. We present superalgerba versions of stable envelopes for the new quiver-like varieties that generalize the cotangent bundle of a Grassmannian. We define superalgebra generalizations of the Tarasov-Varchenko weight functions, and show that they represent the super stable envelopes. Both super stable envelopes and super weight functions transform according to Yangian $\check{R}$-matrices of $\glsv{M}{N}$ with $M+N=2$.
\end{abstract}

\maketitle

\def\gln{\mathfrak{gl}(\zN)}
\def\zN{ N }


\section{Introduction}

There is a well-known correspondence between an $A_n$-type framed Nakajima quiver variety and a weight space in the tensor product of fundamental representations of $\gln$. This correspondence is used for the geometric construction of Yangian $R$-matrices and for the categorification of the quantum group $\GLqn$.  In recent works of Okounkov and his co-authors \cite{MO, O, AO, Onew}  the key ingredient of this correspondence is the collection of so-called stable envelope maps.

The purpose of this paper is to extend this correspondence from $\gln=\glsv{\zN}{0}$ to Lie superalgebras $\glmn$.

In Section \ref{sec:newfamily} we introduce a new family of quiver-related varieties, defined by a modified Nakajima-Cherkis construction.
In the original construction the Nakajima-Cherkis quiver variety (a.k.a. bow variety) is a Hamiltonian reduction of the product of edge-related symplectic varieties $\xXsme$ by the product of vertex-related groups $\GLnv$. We show that the same quiver variety can be presented as an intersection of edge-related generalized Lagrangian subvarieties $\tcLvr$ in the product of cotangent bundles $\rmTs \glnv$. The superalgebra related varieties result from replacing some generalized varieties $\tcLvr$ by their Legendre transforms.  A replacement of some Nakajima arrow edges turns $\glN$ into $\glmn$, while a replacement of some Cherkis bow edges turns the corresponding fundamental representations of $\glmn$ into their parity-flipped twins.

Consider a tensor power of the defining vector representation of $\gl(N)$. The quiver variety corresponding to a weight space of this representation is a cotangent bundle of an $\zN$-step partial flag variety. In particular, for $\zN=2$, it is the cotangent bundle of a Grassmannian. Starting in Section \ref{secr} we carry out detailed calculations showing the four possible Legendre transform generalizations of this case. The generalized varieties, associated with the four decorated quivers of \eqref{thefourquivers} below, will still be total spaces of vector bundles over the Grassmannian, but of different bundles---see Figure \ref{fig1}. We define the superalgebra generalization of Maulik-Okounkov stable envelopes (`super stable envelopes'), and show their existence using the superalgebra generalization of Tarasov-Varchenko weight functions (`super weight functions'). We show that both super stable envelopes and super weight functions transform according to the Yangian $\check{R}$-matrices of
\[
\gl(\C_{\even} \oplus \C_{\even}), \quad
\gl(\C_{\even} \oplus \C_{\odd}), \quad
\gl(\C_{\odd} \oplus \C_{\even}), \quad
\gl(\C_{\odd} \oplus \C_{\odd}).
\]

The quiver-related varieties coming from Legendre-trasformed arrow edges appeared in the work~\cite{OR} of A.~Oblomkov and the second author on link homology, where either of two types of  the fundamental representation of $\GL$ family were assigned to each link component. The categorical representation of the braid group amounted to the categorification of the $\glmn$ Hecke algebra, where $M$ and $N$ are the numbers of braid strands colored with either type of the fundamental representation of $\GL$. Upon the reduction to $\glv{n}$ or, more generally, to $\glv{m|n}$ homology, one of these representations becomes the fundamental representation of $\glv{n}$ or $\glv{m|n}$, while the other becomes its parity-flipped twin.

\medskip

\noindent{\bf Acknowledgments.} L.R. wants to thank T.~Dimofte, R.~Bezrukavnikov, I.~Losev and especially A.~Oblomkov for many discussions and patient explanations of the basics of quiver varieties and their relation to string physics.
R.R. is supported by the Simons Foundation grant 523882.
The work of L.R. was supported in part by the NSF grant DMS-1108727.

\section{A new family of quiver varieties}\label{sec:newfamily}

\subsection{A Nakajima-Cherkis quiver variety}

In this section 
we recall the definition of varieties associated to quivers with two kinds of edges: arrow edges and bow edges. These varieties have also been called bow varieties (\cite{Ch09, Ch10, Ch11, NT, Nakajima_Satake, RS}), but in this paper we employ the methaphor that a {\em ``quiver''} can hold both {\em ``arrows''} and {\em ``bows''} so we keep calling the varieties with arrow and bow components quiver varieties. The history of quiver varieties without bow edges goes back to \cite{NakajimaOriginal, NakajimaDuke}, for a more recent survey see \cite{Ginzburg}.


\subsubsection{Hamiltonian reduction} \label{211}
In this paper we consider only linear quivers. Thus, a quiver $\xQ$ is a `linear' graph with two univalent vertices, $\NQ-1$  bi-valent vertices and $\NQ$ edges:
\[
\begin{tikzpicture}[baseline=-0.15cm]
\draw[o-o] (0,0)-- (1.5,0);
\draw[-o] (1.5,0) -- (3,0);
\draw[o-] (-3,0) -- (-2.25,0);
\draw[-] (-0.75,0) -- (0,0);
\draw[-] (3,0) -- (3.75,0);
\draw[-o] (5.25,0) -- (6,0);
\node at (-1.5,0) {$\cdots$};
\node at (4.5,0) {$\cdots$};
\node[below] at (0,0) {$\xmimo$};
\node[below] at (1.5,0) {$\xmi$};
\node[below] at (3,0) {$\xmio$};
\node[below] at (-3,0) {$\xmz$ };
\node[below] at (6,0) {$\xmN$};
\end{tikzpicture},
\]
$\xmi$ being a vertex index.
A vertex
$\xmi$ is assigned a non-negative integer $\xni$ representing the $\xni$-dimensional vector space $\xVi$ and the group $\Gi = \GL(\xVi)$. We always assume that $\xnz = \xnv{\NQ} = 0$.



An oriented edge
connecting the vertices $\xmimo$ and $\xmi$ is assigned a symplectic variety $(\xXsmi,\xom)$ with the Hamiltonian action of $\Gimo\times\Gi$ and the corresponding moment maps $\murv{i-1}$ and $\mulv{i}$.
If the orientation of the edge
is reversed, then $\xom$, $\murv{i-1}$ and $\mulv{i}$ change signs, but $(\xXsmi,\xom)$ and $(\xXsmi,-\xom)$ are symplectomorphic, because $\xXsmi$ is (a Hamiltonian reduction of) a cotangent bundle, and one can switch the sign of cotangent fibers. Hence ultimately the choice of orientation of the edges does not affect the resulting quiver variety (unless the orientation is also used to specify the stability conditions).

The  quiver variety $\xXQ$ is a result of the Hamiltonian reduction of the product of edge varieties with respect to all vertex groups:
\begin{equation}
\label{eq:bquiv}
\xXQ := \hXQ\big|_{\xmui=0\atop i=0,\ldots,\NQ}
\hmrd \GQ,
\end{equation}
where
\[
\hXQ =
\prod_{i=1}^{\NQ}
\xXsmi,\quad
\GQ =
\prod_{i=0}^{\NQ}
\GLVi,
\]
while $\xmui$ is the total moment map of the vertex $\xv$:  \[\xmui =
\mulv{i} + \murv{i}.\]

\subsubsection{Arrow and bow edges}\label{212}
The edges of a quiver are of two types: an \aedg\ and a \bedg:
\[
\begin{tikzpicture}
\draw[o-o] (0,0)--(1.5,0);
\end{tikzpicture}
:\quad
\begin{tikzpicture}
\draw[o-o,thick] (0,0)--(1.5,0);
\node[above] at (0.75,0) {arrow};
\end{tikzpicture}
\;,\qquad
\begin{tikzpicture}
\draw[o-o,thick, dashed] (0,0)--(1.5,0);
\node[above] at (0.75,0) {bow};
\end{tikzpicture}
\]
The corresponding symplectic varieties $\xXsmi$ are also of two types: the \avrs\ $\varrnm$ and the \bvrs\ $\feslnm$, where $m$ and $n$ are non-negative integers representing the dimensions of adjacent vertex spaces. An \avr\ $\varrnm$ has a Hamiltonian action of $\GLm\times\GLn$, while a \bvr\ $ \feslnm$ has a Hamiltonian action of $\GLm\times\GLn\times\Cs$. The groups $\Cs$ acting on \bvrs\ combine into the group \[\Csbow = \prod_{\text{$\edg$ \scriptsize is bow}} \Cse\] acting on the quiver variety $\xXQ$.

\subsubsection{The \avr}\label{213}
For two non-negative integers $m,n$ define
\[
\varrnm = \rmTs \Hom(\C^m,\C^n)
\]
with the natural Hamiltonian action of $\GL_m\times \GL_n$. The moment maps are $\xmunm = - YX$ and $\xmunn=XY$, where $(X,Y) \in \Hom(\C^m,\C^n)\times\Hom(\C^n,\C^m) = \rmTs\Hom(\C^m,\C^n)$.

For geometrizing the Yangian $R$-matrix we need a `stable' version of the \avr. Namely, denote
$\Homst(\C^m,\C^n)\subset \Hom(\C^m,\C^n)$ the set of linear maps of the highest rank and define
\begin{equation}
\label{eq:stvar}
\varrstnm = \rmTs \Homst(\C^m,\C^n).
\end{equation}
Now define the symplectic variety of an oriented \aedg\ as
\[
\begin{tikzpicture}[baseline=-0.15cm]
\draw[o->, thick] (0,0) -- (0.8,0);
\draw[-o,thick] (0.8, 0) -- (1.5,0);
\node[below] at (0,0) {$\xmimo$};
\node[below] at (1.5,0) {$\xmi$};
\end{tikzpicture}
,\qquad
\xXsmi = \varrv{\xni}{\xnimo}\quad\text{or}\quad \varrstv{\xni}{\xnimo}.
\]

\subsubsection{The \bvr, cf. \cite[Section 2]{Nakajima_draft}} \label{214}
\def\yno{n} \def\ynt{m}
%

Denote $U_k\subset\GLv{k}$ the subgroup of upper-triangular unipotent matrices.
For two non-negative integers $m\leq n$ define the subgroup $\Unm\subset U_n$ consisting of upper-triangular unipotent matrices of the form
\[
h=
\begin{pmatrix}[c|c]
u & \ast
\\
\hline
0 & I
\end{pmatrix},
\]
where $u\in U_{\yno-\ynt}$, while
$\ast$ is any $\ynt\times(\yno-\ynt)$ matrix and $I$ is the $\ynt\times\ynt$ identity matrix.

We define the action of $\Unm$ on $\GLn\times\C^\ynt$.
Let $\Unm$ act on $\GLn$ by right multiplication.
Denote $w(h)$ the last row of the matrix $\ast$. Then $w(h_1 h_2) = w(h_1) + w(h_2)$ and we define the action of $\Unm$ on $\C^{\ynt}$ as $h\cdot v = w(h) + v$. Now $\feslnm$ is the `twisted' symplectic quotient:
\[
\feslnm=\rmTs\bigl(\GLn\times\C^{\ynt}\bigr)\hmrd_{x_{\yno,\ynt}} \Unm :=
\rmTs\bigl(\GLn\times\C^{\ynt}\bigr)
\Big|_{\mu_{\Unm}=x_{\ynt,\yno}}/\Unm,
\]
where
\[
x_{\ynt,\yno} =
\begin{pmatrix}[c|c]
x_0 & 0
\\
\hline
0 & 0
\end{pmatrix},
\]
and $x_0$ is the $(\yno-\ynt) \times (\yno-\ynt)$ transposed nilpotent Jordan block.
Now we define the symplectic variety of a \bedg\ as
\[
\begin{tikzpicture}[baseline=-0.15cm]
\draw[o-o,thick,dashed] (0,0) -- (1.5,0);
\node[below] at (0,0) {$\xmimo$};
\node[below] at (1.5,0) {$\xmi$};
\end{tikzpicture}
,\qquad
\xXsmi =
\begin{cases}
\feslv{\xnimo}{\xni},&\text{if $\xnimo\geq \xni$,}
\\
\feslv{\xni}{\xnimo},&\text{if $\xnimo\leq \xni$.}
\\
\end{cases}
\]
%

The \bvr\ $\feslnm$ has a Hamiltonian action of the group $\GLn\times\GLm\times \Cs$ stemming from its action on $\GLn\times\Cm$. The group $\GLn$ acts on $\GLn$ by left multiplication and it does not act on $\Cm$; while $\GLm\times\Cs$ acts on $\GLn$ by right multiplication: $(h,z)\cdot g = g M^{-1}(h,z)$, where
\[
M(h,z) =
\begin{pmatrix}[c|c]
zI & 0
\\
\hline
0 & h
\end{pmatrix}.
\]
Finally, $\GLm\times\Cs$ acts on $\Cm$ by natural action and scaling: $(h,z)\cdot x = z\,hx$.

\subsubsection{Edge charges and the Hanany-Witten move}
Consider a \xlq:
\[
\begin{tikzpicture}
\draw[o-o] (0,0) -- (1.5,0);
\draw[-o] (1.5,0) -- (3,0);
\draw[o-o] (4.5,0) -- (6,0);
\draw[-o] (6,0) -- (7.5,0);
\node at (3.75,0) {$\cdots$};
\node[below] at (0,0) {$0$};
\node[below] at (1.5,0) {$n_1$};
\node[below] at (3,0) {$n_2$};
\node[below] at (7.5,0) {$0$};
\node[below] at (6,0) {$n_{N}$};
\node[below] at (4.5,0) {$n_{N-1}$};
\end{tikzpicture}
\]
with an arbitrary distribution of arrow and bow edges.

We always assume that the leftmost and the rightmost vertices are assigned number 0. For an edge $\edg$ connecting a vertex $\xmimo$ on the left and $\xmi$ on the right
define $\xcnLv{\edg}$ as the number of edges of opposite type to the left of it and $\xcnRv{\edg}$ -- to the right of it.
Define the charge $\xcne$ of $\edg$:
\[
\xcne =
\begin{cases}
\xni - \xnimo  + \xcnLv{\edg}, & \text{if $\edg$ is an arrow,}
\\
\xnimo - \xni + \xcnRv{\edg}, &\text{if $\edg$ is a bow.}
\end{cases}
\]
Since the number $\xnz$ at the leftmost vertex is fixed (zero), the charges of the edges determine the numbers at all vertices of the quiver. We consider only the quivers for which $\xcne\geq 0$ for all edges $\edg$ and all $\xni\geq 0$ for all vertices $\xmi$. The latter condition imposes a constraint on possible edge charge assignments.

The Hanany-Witten  move  \cite{HW} transposes two neighboring edges of opposite nature: the quiver variety resulting from this move is isomorphic to the original one as long as the charges of edges are preserved:
\begin{equation}
\label{eq:hwmvp}
\begin{tikzpicture}[baseline=-0.15cm]
\draw[o-,thick] (0,0) -- (1.5,0);
\draw[o-o,thick, dashed] (1.5,0) -- (3,0);
\node[below] at (0,0) {$\xnimo$};
\node[below] at (1.5,0) {$\xni$};
\node[below] at (3,0) {$\xnio$};
\end{tikzpicture}
\quad
\longleftrightarrow
\quad
\begin{tikzpicture}[baseline=-0.15cm]
\draw[o-,thick,dashed] (0,0) -- (1.5,0);
\draw[o-o,thick] (1.5,0) -- (3,0);
\node[below] at (0,0) {$\xnimo$};
\node[below] at (1.5,0) {$\xni'$};
\node[below] at (3,0) {$\xnio$};
\end{tikzpicture},
\quad\text{if
$\xni+\xni' = \xnimo+\xnio-1.$
}
\end{equation}
If the middle vertex number after the transposition ($\xni$ or $\xni'$) has to be negative, then the original quiver variety is empty.

\subsubsection{A \sprd\ quiver and its variety}
Consider a \xlq\ with $\gN$ \aedg s with charges $\bfw = (\xw_1,\ldots, \xw_{\gN})$ and $\rM$  \bedg s with charges $\bfhw=(\hw_1,\ldots,\hw_{\rM})$. Denote $\xab{\bfw} = \sum_{i=1}^{\gN} \xw_i $ and similarly for $\xab{\bfhw}$.
Since the numbers at end-point vertices are zero, the charges must satisfy the consistency condition:
$\xab{\bfw} = \xab{\bfhw}$.
If this condition is satisfied, then there exists a \emph{\sprd} quiver
$\xQbwk$ in which all  $\gN$ \aedg s are on the left and all $\rM$ \bedg s on the right:
\begin{equation}
\label{eq:spquiv}
\xQbwk:\qquad
\begin{tikzpicture}[scale=1,baseline=-0.1cm]
\draw[o-o,thick] (1.5,0) -- (3,0);
\draw[o-o,thick] (4.5,0) -- (6,0);
\draw[-o,dashed, thick] (6,0) -- (7.5,0);
\draw[o-o,dashed,thick] (9,0) -- (10.5,0);
\node at (3.75,0) {$\cdots$};
\node[below] at (1.5,0) {$0$};
\node[below] at (3,0) {$n_1$};
\node[below] at (7.5,0) {$n_{N+1}$};
\node[below] at (6,0) {$\nmax$};
\node[below] at (4.5,0) {$n_{N-1}$};
\node[below] at (9,0) {$n_{N+\rM-1}$};
\node[below] at (10.5,0) {$0$};
\node at (8.25,0) {$\cdots$};
\node[above] at (2.25,0) {$\xw_1$};
\node[above] at (5.25,0) {$\xw_{\gN}$};
\node[above] at (6.75,0) {$\hw_{\rM}$};
\node[above] at (9.75,0) {$\hw_1$};
\end{tikzpicture},
\end{equation}
and $\xnv{\gN} = \nmax$, where $\nmax:= \xab{\bfw} = \xab{\bfhw}$ is the number at the \emph{middle vertex} which separates the arrow and bow parts of the quiver.
Since the edge charges are non-negative, the vertex numbers are in relation
\[
0\leq n_1\leq \cdots\leq n_{N-1}\leq n_N,\qquad n_N \geq n_{N-1} \geq \cdots\geq n_{N+\rM-1}\geq 0.
\]

Defining the variety $\xXbwk $ associated with the \sprd\ quiver $\xQbwk$, we use varieties $\varrstnm$ of~\eqref{eq:stvar} for \aedg s.

We split the \sprd\ quiver into the arrow and bow halves:
\begin{equation}
\label{eq:splq}
\xQaw:\quad
\begin{tikzpicture}[scale=1,baseline=-0.1cm]
\draw[o-o,thick] (1.5,0) -- (3,0);
\draw[o-,thick] (4.5,0) -- (5.85,0);
\node at (3.75,0) {$\cdots$};
\node[below] at (1.5,0) {$0$};
\node[below] at (3,0) {$n_1$};
\node[below] at (6,0) {$\nmax$};
\node at (6,0) {$\Box$};
\end{tikzpicture}
,\qquad
\xQbw:\quad
\begin{tikzpicture}[scale=1,baseline=-0.1cm]
\node at (0,0) {$\Box$};
\draw[-o,dashed,thick] (0.15,0) -- (1.5,0);
\node at (2.25,0) {$\cdots$};
\draw[o-o,dashed,thick] (3,0) -- (4.5,0);
\node[below] at (0,0) {$\nmax$};
\node[below] at (3,0) {$\xnv{\gN+K-1}$};
\node[below] at (4.5,0) {$0$};
\end{tikzpicture}
\end{equation}
The boxes at end-vertices indicate that we do not perform the Hamiltonian reduction there.

The  arrow quiver variety
is the cotangent bundle to a partial flag variety: $\rmTs\flFw$, where
$\flF_{\bfw} = \{\Fb \}$ and
\[
\Fb = (F_0\subset F_1\subset \cdots \subset F_{\gM-1}\subset F_{\gM}=\C^{|\bfw|}
),\qquad \dim F_0 = 0,\quad \dim F_{i+1} - \dim F_{i} = \xw_i.
\]

We denote the bow variety by $\xSSbhw$.  If the \bedg\ charges are non-decreasing:
\[
\hw_1 \leq \cdots\leq \hw_{\rM},
\]
then the bow variety $\xSSbhw$ is the equivariant Slodowy slice introduced by I.~Losev~\cite{L}. Denote $\xSSobhw$ the Slodowy slice corresponding to the nilpotent matrix with Jordan block decomposition given by $\bfhw$. Then $\xSSbhw = \GLv{\nmax}\times \xSSobhw$ and the moment map for the action of $\GLv{\nmax}$ is
$\xmuS(g,x) = \Adv{g} x$.

In order to describe the action of $\Csbow = (\Cs)^{\rM}$ on $\xSSbhw$, we split $\C^{\nmax} =
\C^{\hw_1}\oplus\cdots\oplus\C^{\hw_{\rM}}$. For $\bfz=(z_1,\ldots,z_{\rM})\in\Csbow$ denote $M_{\bfz}$ the diagonal matrix which multiplies each component $\C^{\hw_i}$ by $z_i$. Then $\bfz\cdot(g,x) = (g \Mbz^{-1},\Adv{\Mbz} x)$.


Since the \sprd\ quiver $\xQbwk$ results from joining $\xQaw$ and $\xQbw$ at the middle vertex, the corresponding variety is the Hamiltonian reduction with respect to the middle group:
\begin{equation}
\label{eq:glqv}
\xXbwk = (\rmTs\xFLbw\times \xSSbhw)\hmrd \GLv{\nmax}.
\end{equation}

\subsubsection{The $\GLN$ weight space quiver}
A weight of a $\GLN$-module is determined by $\gN$ ordered non-negative integers $\bfw = (\xw_1,\ldots, \xw_{\gN})$. Denote $\xRk$ the $\xk$-th fundamental representation of $\GLN$, that is $\xRo$ is the defining $\gN$-dimensional module and $\xRk = \xWd^{\xk} \xRo$. For an ordered sequence of non-negative integers
$\bfhw=(\hw_1,\ldots,\hw_{\rM})$ denote
\[
\xRbk = \xRv{\hw_1}\otimes\cdots\otimes\xRv{\hw_\rM}
\]
and denote $\wVbwk\subset \xRbk$ its weight space of weight $\bfw$. The corresponding quiver is a \xlq\ consisting of $\rM$  \bedg s with charges $\bfhw$ and $\gN$ \aedg s with charges $\bfw$. The edges can be distributed randomly along the quiver. In this paper we use the quiver $\xQbwk$ of~\eqref{eq:spquiv} and its variety $\xXbwk$ of~\eqref{eq:glqv}.

In particular, if we consider the tensor product of only defining representations $\xRo\otimes\cdots\otimes\xRo$, then
$\bfhw = \bfo = (1,\ldots,1)$, and $\xSSbo = \rmTs \GLnmax$, so the corresponding variety is the cotangent bundle to a partial flag variety:
$\xXbko = \rmTs\xFLbw$

\subsection{Alternative construction}
For our generalization in Section \ref{sec:newvars} we need an alternative construction of arrow-bow quiver varieties, which we describe now.

\subsubsection{Critical locus}

Let $\xXsm$ be a symplectic variety with the Hamiltonian action of a Lie group  $\xxG$ and the corresponding moment map $\xmu$. The adjoint action of $\xxG$ on its Lie algebra $\xxg$ extends to the action of $\xxG$ on $\xXsm\times \xxg$.
Consider a $\xxG$-invariant function $\xWsm \in \C[\xXsm\times\xxg]^{\xxG}$ defined as a pairing of $\xmu$ and the elements of $\xxg$: $\xWsm(x,X) = \Tr \xmu(x) X.$ If the action of $\xxG$ on $\xXsm$ is free, then the projection $\xXsm\times\xxg\longrightarrow \xXsm$ establishes an isomorphism between the critical locus 
$\Critvv{\xWsm}{\xXsm\times\xxg}$ of $\xWsm$ on $\xXsm\times\xxg$
and the subvariety $\xXsm|_{\xmu=0}$. As a result, the Hamiltonian reduction of $\xXsm$ can be presented as a (GIT) quotient of the critical locus of $\xWsm$:
\begin{equation}
\label{eq:smcrit}
\xXsm \hmrd \xxG \cong
\Critvv{\xWsm}{\xXsm\times\xxg}/\xxG.
\end{equation}

\subsubsection{Symplectic intersection}
For a given Lie group $\xxG$ we consider `\Gpr s' $(\xX,\xW)$, where $\xX$ is a variety with the $\xxG$ action and $\xW$ is a $\xxG$-invariant function on $\xX\times\xxg$: $\xW\in\C[\xX\times\xxg]^{\xxG}$. In all our examples $\xW$ is linear as a function on $\xxg$, that is, there is a function $\xmu\colon \xX\rightarrow \xxg$ (not necessarily a moment map) and $W = \Tr \xmu X$. For two \Gpr s $(\xX_i,\xW_i)$, $i=1,2$ we define their symplectic intersection as the critical locus:
\begin{equation}
\label{eq:defsint}
(\xX_1,\xW_1) \smcap (\xX_2,\xW_2) := \Critvv{\xW_2 - \xW_1}{\xX_1 \times\xX_2\times\xxg}.
\end{equation}

This intersection has a symplectic geometry interpretation. A pair $(\xX,\xW)$ determines a `generalized' $\xxG$-invariant Lagrangian subvariety of $\rmTs\xxg$ or, equivalently, a generalized Lagrangian subvariety of the Hamiltonian reduction $\rmTs\xxg\hmrd \xxG$. Present $\rmTs\xxg = \xxg\times \xxg$ with coordinates $(X,Y)$.
 By definition,
\[
\tcLXW := \left\{ (x,X,Y) \in \xX\times\rmTs\xxg \; \Big| \; \frac{\partial \xW}{\partial X } = Y,\; \frac{\partial \xW}{\partial x} = 0\right\}.
\]
The image $\cLXW\subset \rmTs\xxg$ of $\tcLXW$ under the projection $\xX\times\rmTs\xxg\longrightarrow\rmTs\xxg$ is a (possibly singular) Lagrangian subvariety of $\rmTs\xxg$. Thus the generalized Lagrangian subvariety $\tcLXW$ represents a fibration $\tcLXW\rightarrow\cLXW$ with a Lagrangian base. We consider two \Gpr s equivalent: $\xGprv{\xX_1}{\xW_1}\sim \xGprv{\xX_2}{\xW_2}$, if they produce the same fibration.

Define the intersection of two generalized Lagrangian subvarieties as the product of fibers over the intersection of their bases:
\begin{multline*}
\tcLXWv{1} \lgcap \tcLXWv{2} :=
\\
\{
(x_1,x_2,X,Y)\in\xX_1\times\xX_2\times\rmTs\xxg\;|\;
(x_1,X,Y)\in \tcLXWv{1},\;(x_2,X,Y)\in\tcLXWv{2}
\}.
\end{multline*}

Now a projection $
\xX_1\times\xX_2\times \rmTs\xxg \longrightarrow
\xX_1\times\xX_2\times\xxg$ identifies the symplectic intersection of pairs with the intersection of their generalized Lagrangian subvarieties:
\[
\tcLXWv{1}\lgcap\tcLXWv{2} \xrightarrow{\;\;\cong\;\;}
(\xX_1,\xW_1) \smcap (\xX_2,\xW_2).
\]
\subsubsection{Brief 2-category motivation}
\Gpr s represent objects in the 2-category~\cite{KRS} associated with the hamiltonian quotient  $\rmTs \xxg \hmrd \xxG$ considered as a symplectic variety. The category of morphisms between two \Gpr s $(\xX_1,\xW_1)$ is the category of $\xxG$-equivariant  matrix factorizations of $\xW_2 - \xW_1$ over $\xX_1 \times\xX_2\times\xxg$. This category is `approximately' equivalent to the derived category of $\xxG$-equivariant coherent sheaves over the critical locus~\eqref{eq:defsint}, which motivates the set-theoretical definition of the symplectic intersection.

The particular 2-category of $\rmTs \xxg \hmrd \xxG$  and its arrow edge-related objects were studied in detail in~\cite{OR1} in relation to the categorical representation of the braid group and the construction of the link homology.


\subsubsection{Quiver varieties as symplectic intersections}
The relation~\eqref{eq:smcrit} allows us to transform the standard definition~\eqref{eq:bquiv} of the quiver variety $\xXQ$ into the symplectic intersection~\eqref{eq:defsint}. For an edge $\xe$ connecting the vertices $\xvo$ and $\xvt$, its edge variety $\xXsme$ becomes a pair $(\xXsme;\xWsme)$, where $\xWsme = \Tr \muvoe \Xvoe +  \Tr\muvte\Xvte$, relative to the Lie algebra $\glv{\nvo}\times\glv{\nvt}$. Now the quiver variety $\xXQ$ can be presented as a quotient of the symplectic intersection of all pairs $(\xXsme;\xWsme)$ in the total Lie algebra
$\gQ= \prod_{\vrt\in Q_{\rmv}} \glnv$
\[
\xXQ = \overset{\ism}{\bigcap_{\xe\in Q_{\rme}}}\xGprv{\xXsme}{\xWsme} \Big/ \GQ
:= \Critvv
{ \xWsmet}
{\hXQ\times\gQ}/\GQ,
\]
where
\[
\xWsmet =  \sum_{\xe\in Q_{\rme}}\xWsme = \sum_{\xv\in Q_{\rmv}} \Tr \xmuv X_{\xv}.
\]

\subsection{New quiver-related varieties}\label{sec:newvars}
\subsubsection{A Legendre transform}
For a \Gpr\ $\xGprv{\xX}{\xW}$ define a Legendre-transformed pair as
\[
\Lgv{\xGprv{\xX}{\xW}} := \xGprbv{\xX\times\xxg}{-\xW(x,Z) + \Tr XZ},
\]
where
$(x,Z) \in \xX\times\xxg$ and $\xxG$ has adjoint action on $\xxg$. The generalized Lagrangian subvarieties of $\xGprv{\xX}{\xW}$ and $\Lgv{\xGprv{\xX}{\xW}}$ are related by the Legendre anti-symplectomorphism
\[\fLg\colon \rmTs\xxg\longrightarrow\rmTs\xxg,\qquad (X,Y)\mapsto (Y,X).\]
As a consequence, the symplectic intersection of two \Gpr s is isomorphic to the symplectic intersection of their Legendre transforms:
\[
\Lgv{\xGprv{\xX_1}{\xW_1}} \smcap \Lgv{\xGprv{\xX_2}{\xW_2}}
\cong(\xX_1,\xW_1) \smcap (\xX_2,\xW_2).
\]
Finally, $\fLg^2 = 1$, that is, the double Legendre transform of a \Gpr\ is equivalent to the original pair:
\[
\Lgv{\bigl(\Lgv{\xGprv{\xX}{\xW}}\bigr)} \sim \xGprv{\xX}{\xW}.
\]
\subsubsection{Legendre transform and quiver varieties}
The Legendre transform can be applied to a \Gpr\ $\xGprv{\xXsme}{\xWe}$ associated with an edge $\edg$ of a quiver. If the edge $\edg$ is attached to a vertex $\xv$, then we define the one-sided transform
\[
\Lgvv{\xGprv{\xXsme}{\xWe}}{\xv} :=
\xGprv{\xXsme\times\xxg_{\xv}}{\Tr Z_{\xv}(X_{\xv} - \xmuv)}.
\]
The two-sided transform $\Lgv{\xGprv{\xXsme}{\xWe}}$ is defined as the application of one-sided transforms on both sides of the edge $\edg$.

A \xmrk\ quiver $Q$ has marks ($\ast$) at the ends of some of its edges. A mark means that the \Gpr\ of the edge is Legendre-transformed at that side. Thus, depending on the marks, a \Gpr\ $\xGprv{\xXe}{\xWe}$ of an edge $\edg$ attached to the vertices $\xv_1,\xv_2$ may be of one of the four forms:
\begin{align*}
\xGprv{\xXsme}{\xWsme}:&
\begin{tikzpicture}[baseline=-0.15cm]
\draw[o-o] (0,0) -- (1.5,0);
\node[below] at (0,0) {$\xv_1$};
\node[below] at (1.5,0) {$\xv_2$};
\end{tikzpicture},
&
\Lgvv{\xGprv{\xXsme}{\xWsme}}{\xv_1}:&
\begin{tikzpicture}[baseline=-0.15cm]
\draw[o-o] (0,0) -- (1.5,0);
\node[below] at (0,0) {$\xv_1$};
\node[below] at (1.5,0) {$\xv_2$};
\node[above] at (0.25,0) {$\ast$};
\end{tikzpicture},
\\
\Lgvv{\xGprv{\xXsme}{\xWsme}}{\xv_2}: &
\begin{tikzpicture}[baseline=-0.15cm]
\draw[o-o] (0,0) -- (1.5,0);
\node[below] at (0,0) {$\xv_1$};
\node[below] at (1.5,0) {$\xv_2$};
\node[above] at (1.25,0) {$\ast$};
\end{tikzpicture},
&
\Lgv{\xGprv{\xXsme}{\xWsme}}: &
\begin{tikzpicture}[baseline=-0.15cm]
\draw[o-o] (0,0) -- (1.5,0);
\node[below] at (0,0) {$\xv_1$};
\node[below] at (1.5,0) {$\xv_2$};
\node[above] at (0.25,0) {$\ast$};
\node[above] at (1.25,0) {$\ast$};
\end{tikzpicture} =
\begin{tikzpicture}[baseline=-0.15cm]
\draw[o-o] (0,0) -- (1.5,0);
\node[below] at (0,0) {$\xv_1$};
\node[below] at (1.5,0) {$\xv_2$};
\node[above] at (0.75,0) {$\ast$};
\end{tikzpicture},
\end{align*}
that is, a single mark in the middle means a complete (two-sided) Legendre transform.

Note that a mark can be moved from one edge to the other at the same vertex and if two edges are marked at the same vertex, then these marks can be removed:
\begin{equation}
\label{eq:dblL}
\begin{tikzpicture}[baseline=-0.1cm]
\draw[o-] (0,0) -- (0.75,0);
\draw[-] (-0.75,0) -- (0,0);
\node[above] at (-0.25,0) {$\ast$};
\end{tikzpicture}
=
\begin{tikzpicture}[baseline=-0.1cm]
\draw[o-] (0,0) -- (0.75,0);
\draw[-] (-0.75,0) -- (0,0);
\node[above] at (-0.25,0) {$\ast$};
\end{tikzpicture},\qquad
\begin{tikzpicture}[baseline=-0.1cm]
\draw[o-] (0,0) -- (0.75,0);
\draw[-] (-0.75,0) -- (0,0);
\node[above] at (-0.25,0) {$\ast$};
\node[above] at (0.25,0) {$\ast$};
\end{tikzpicture}
=
\begin{tikzpicture}[baseline=-0.1cm]
\draw[o-] (0,0) -- (0.75,0);
\draw[-] (-0.75,0) -- (0,0);
\end{tikzpicture}.
\end{equation}

\subsubsection{Mixed vector bundles over partial flag varieties}
As an example of the latter construction, consider the following quiver $Q$:
\begin{equation}
\label{eq:qst}
\begin{tikzpicture}
\draw[o-o, thick] (0,0) -- (1.5,0);
\draw[-o,thick] (1.5,0) -- (3,0);
\draw[-o,thick] (3,0) -- (4.5,0);
\node at (5.25,0) {$\cdots$};
\draw[o-o,thick] (6,0) -- (7.5,0);
\draw[-,thick] (7.5,0) -- (8.85,0);
\node at (9,0) {$\Box$};
\node[above] at (2.25,0) {$\ast$};
\node[above] at (3.75,0) {$\ast$};
\node[above] at (8.25,0) {$\ast$};
\node[below] at (0.75,0) {$\xw_1$};
\node[below] at (2.25,0) {$\xw_2$};
\node[below] at (3.75,0) {$\xw_1$};
\node[below] at (6.75,0) {$\xw_{\gM-1}$};
\node[below] at (8.25,0) {$\xw_{\gM}$};
\end{tikzpicture}
\end{equation}
All of its edges are of the arrow type, their charges being the non-negative integers $\bfw = (\xw_1,\ldots, \xw_{\gM})$. The marks are distributed randomly among the edges. The box $\Box$ at the end of the quiver indicates that we do not apply \sint\ with respect to its Lie algebra.

The resulting variety has the following description. Consider a partial flag variety
$\flF_{\bfw} = \{\Fb \}$, where
\begin{equation}\label{eq:filtration}
\Fb = (F_0\subset F_1\subset \cdots \subset F_{\gM-1}\subset F_{\gM}=\C^{|\bfw|}
),\qquad \dim F_0 = 0,\quad \dim F_{i+1} - \dim F_{i} = \xw_i.
\end{equation}
For a partial flag $\Fb$ consider a subspace $\xsVFb\subset \End(\C^{|\bfw|})$ such that $\xphi\in \xsVFb$ iff
\begin{equation}\label{eq:part}
\xphi(F_{i}) \subset
\begin{cases}
F_{i},&\text{if the $i$-th edge is marked},
\\
F_{i-1},&\text{if the $i$-th edge is unmarked},
\end{cases}
\end{equation}
see
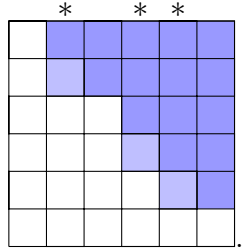
\[
\begin{tikzpicture}[scale=.5,baseline=0]
\draw[fill=blue!40] (6,1) to (6,6) to (1,6) to (1,5) to (2,5) to (2,4) to (3,4) to (3,3) to (4,3) to (4,2) to (5,2) to (5,1);
\draw[fill=blue!25] (1,5) to (2,5) to (2,4) to (1,4);
\draw[fill=blue!25] (3,3) to (4,3) to (4,2) to (3,2);
\draw[fill=blue!25] (4,2) to (5,2) to (5,1) to (4,1);
\draw (0,0) to (6,0) to (6,1) to (0,1) to (0,2) to (6,2) to (6,3) to (0,3) to (0,4) to (6,4) to (6,5) to (0,5) to (0,6) to (6,6) to (6,0) to (5,0) to (5,6) to (4,6) to (4,0) to (3,0) to (3,6) to (2,6) to (2,0) to (1,0) to (1,6) to (0,6) to (0,0);
\node at (1.5,6.3) {$*$}; \node at (3.5,6.3) {$*$}; \node at (4.5,6.3) {$*$}; \end{tikzpicture}.
\]
The quiver~\eqref{eq:qst} produces the \Gpr\ $\xGprv{\xXQ}{\xW_Q}$, where $\xX_Q$ is the bundle over $\flF_{\bfw}$ with fibers $\xsVFb$, while $\xW_Q = \Tr \xphi X$.

\begin{remark}
The image of the map $\xmu\colon \xX_Q\longrightarrow \glv{|\bfw|}$, $\xmu(x) = \xphi$ has an explicit description. Denote by $\gMnil$ the number of unmarked edges and let $\bfwnil = (\xwnil_1,\ldots,\xwnil _{\gMnil})$ be the list of the corresponding numbers $\xw_i$ in descending order: $\xwnil_i \geq \xwnil_j$, if $i<j$. Then the image of $\xmu$ consists of matrices $\xphi\in \End(\C^{|\bfw|})$ such that
\[
\dim \ker \xphi^k \geq \sum_{i=1}^{k} \xwnil_i\qquad\text{for all $k = 1,2,\ldots,\gMnil$}.
\]
\end{remark}

\subsection{String theory motivation}\label{sec:string}
\subsubsection{Old quiver varieties}
It is well-known that a quiver variety is the Higgs branch of a 3-dimensional super-Yang-Mills (SYM) theory of the type considered by Hanany and Witten~\cite{HW}. The theory describes the IIB superstring physics of a stack of D3 branes sandwiched between NS5 and D5 branes. The whole brane arrangement is within the 10-dimensional space with coordinates $x_0,\ldots,x_9$, each brane representing an affine subspace parallel to a coordinate subspace.

The 10-dimensional space $\IR^{10}$ is split into a product of subspaces:
\[
\IRte = \IRcmn \times \IRq\times\IRthNS\times\IRthD
\]
All branes are stretched along the common 3-dimensional space $\IRcmn$ and the table~\ref{tb:2b} describes the extra directions of affine subspaces spanned by various branes.
\begin{table}[!h]
\begin{tabular}{|c||c|c|c|}
\hline
type  & D3 & D5 & NS5
\\
\hline
direction &
$\IRq$ & $\IRthD$ & $\IRthNS$
\\
\hline
\end{tabular}
\vspace{0.5cm}
\caption{}
\label{tb:2b}
\end{table}
D3 branes begin and end on D5 and NS5 branes, and their arrangement along $\IRq$ is dual to the quiver $Q$: the transverse NS5 (resp. D5) branes correspond to arrow (resp. bow) edges, while the segments of D3 branes correspond to the vertices of $Q$,  $\xnv{i}$ being the number of D3 branes between the adjacent D5 and NS5 branes, for example:
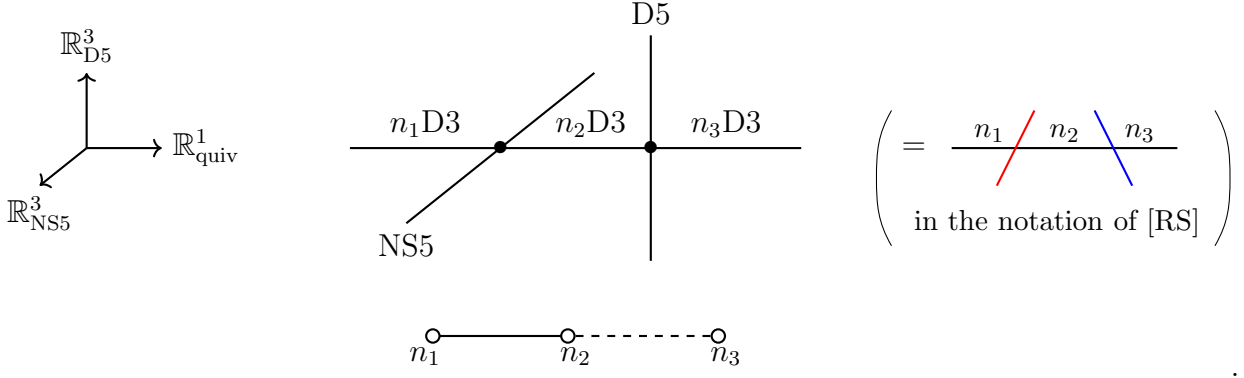
\[
\begin{tikzpicture}
\begin{scope}[xshift=0.5cm]
\draw[thick,->] (-4,0) -- (-3,0);
\node[right] at (-3,0) {$\IRq$};
\draw[thick,->] (-4,0) -- (-4,1);
\node[above] at (-4,1) {$\IRthD$};
\draw[thick,->] (-4,0) -- (-4.625,-0.5);
\node[below] at (-4.625,-0.5) {$\IRthNS$};
\end{scope}
\draw[thick] (0,0) -- (6,0);
\draw[thick] (0.75,-1) -- (3.25,1);
\node [below] at (0.75,-1) {NS5};
\draw[thick] (4,-1.5) -- (4,1.5);
\node[above] at (4,1.5) {D5};
\node[above] at (1,0) {$n_1$D3};
\node[above] at (3.2,0) {$n_2$D3};
\node[above] at (5,0) {$n_3$D3};
\node at (2,0) {$\bullet$};
\node at (4,0) {$\bullet$};
\draw[thick,o-o] (1,-2.5) -- (3,-2.5);
\draw[thick, dashed,-o] (3,-2.5) -- (5,-2.5);
\node[below] at (1,-2.5) {$n_1$};
\node[below] at (3,-2.5) {$n_2$};
\node[below] at (5,-2.5) {$n_3$};
\begin{scope}[xshift=0cm]
\node at (7.5,0){=};
\draw[thick] (8,0)--(11,0) ;
\draw[thick,red] (8.6,-.5)--(9.1,.5);
\draw[thick,blue] (9.9,.5)--(10.4,-.5);
\node at (8.5,.2) {\small $n_1$};
\node at (9.5,.2) {\small $n_2$};
\node at (10.5,.2) {\small $n_3$};
\node at (9.4,-1) {\small in the notation of \cite{RS}};
\draw [thin](7.2,-1.3) to [out=120,in=-90] (7,-.4) to [out=90,in=-120] (7.2,0.5);
\draw [thin](11.5,-1.3) to [out=60,in=-90] (11.7,-.4) to [out=90,in=-60] (11.5,0.5);
\end{scope}
\end{tikzpicture}.
\]

\subsubsection{New quiver varieties}
New quiver-related varieties emerge as Higgs branches of 2d SYM theories describing the physics of a stack of D2 branes sandwiched between NS5 and D4 branes within the IIA string theory. This time the 10-dimensional space-time $\IRte$ is split in the following way:
\[
\IRte = \IRtwcmn \times \IRq\times\IRtwNS\times\IRoD\times\IRtwx\times\IRtwy.
\]
All branes span $\IRtwcmn$. D2 branes are segments along $\IRq$. The branes NS5 span $\IRtwNS$, while the branes D4 span $\IRoD$. Each NS5 (resp. D5) brane may stretch either along $\IRtwx$ or along $\IRtwy$ and depending on this choice, we denote them as \NSx, \NSy\ (resp. \Dx, \Dy). These choices are summed up in the table~\ref{tb:2a}.
\begin{table}[!h]
\begin{tabular}{|c||c|c|c|c|c|}
\hline
type  & D2 & \Dx & \Dy & \NSx &\NSy
\\
\hline
direction &
$\IRq$ & $\IRoD\times\IRtwx$ &$\IRoD\times\IRtwy$& $\IRtwNS\times\IRtwx$ &$\IRtwNS\times\IRtwy$
\\
\hline
\end{tabular}
\vspace{0.5cm}
\caption{}
\label{tb:2a}
\end{table}

The correspondence between the brane arrangements and marked quivers is the same as in the Hanany-Witten IIB construction, except that now the branes \NSx\ and \Dy\ correspond to unmarked edges, while \NSy\ and \Dx\ correspond to marked edges.

If the space $\IRtwx\times\IRtwy$ is endowed with the Taub-NUT metric, then $\IRtwx\times \{0\}$ and $\{0\}\times\IRtwy$ become a pair of cigars and our construction makes contact with that of V.~Mikhaylov and E.~Witten~\cite{MW} who studied the emergence of $\mathrm{U}(M|N)$ Chern-Simons theory when D-branes are wrapped on both cigars. Note however, that we have a skew Howe-dual version here, because in our case the super-algebra is determined by the number of \NSx\ and \NSy\  branes, whereas D4 branes are responsible for its representations.

\subsection{Quiver varieties for $\glmn$ superalgebras} \label{sec24}
\subsubsection{Weights and fundamental representations}
A weight of the superalgebra $\glmn$ is described by two sequences of ordered integers $(\bfw,\bfwp)$, where
$\bfw = (\xw_1,\ldots, \xw_{\gM})$ and
$\bfwp = (\xwp_1,\ldots, \xwp_{\gN})$.

Denote $\xRo$ the defining fundamental representation of $\glmn$: $\xRo\cong \C^{\gM|\gN}= \C^{\gM}_{\mathrm{even}} \oplus \C^{\gN}_{\mathrm{odd}}$ and denote $\xRk = \Lambda^\xk \xRo$. Also denote by $\xRflo$ the parity-flipped fundamental representation: $\xRflo  \cong\C^{\gN|\gM}= \C^{\gM}_{\mathrm{odd}}\oplus \C^{\gN}_{\mathrm{even}}$ and $\xRflk = \Lambda^{\xk} \xRflo$.

For two ordered sequences of non-negative integers
$\bfhw=(\hw_1,\ldots,\hw_{\rM})$ and $\bfhwp=(\hwp_1,\ldots,\hwp_{\rMp})$ denote
\[
\xRbkp = (\xRv{\hw_1}\otimes\cdots\otimes\xRv{\hw_\rM})\otimes
(\xRflv{\hwp_1}\otimes\cdots\otimes\xRflv{\hwp_{\rMp}} )
\]
and denote $\wVbwkp\subset \xRbk$ its weight space of weight $(\bfw;\bfwp)$.

To the weight space $\wVbwkp$ we associate the marked quiver $\xQbwkp$ which is similar to $\xQbwk$ of \eqref{eq:spquiv}. Going from left to right, it has
\begin{enumerate}
\item  $\gN$ marked (that is, Legendre-transformed) \aedg s with charges $\bfhwp$,
\item $\gM$ unmarked (that is, ordinary) \aedg s with charges $\bfhw$,
\item $\rM$ marked \bedg s with charges  $\bfhw$ from right to left,
\item $\rMp$ unmarked \bedg s with charges  $\bfhwp$ from right to left.
\end{enumerate}
Relations~\eqref{eq:dblL} allow us to present this quiver by using only two endpoint marks:
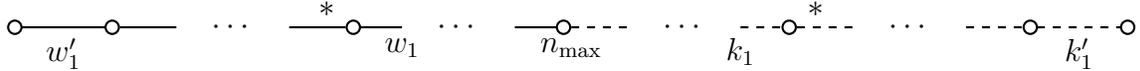
\[
\begin{tikzpicture}
\draw[o-o,thick] (0,0) -- (1.5,0);
\node[below] at (0.75,0) {$\xwp_1$};
\draw[-,thick] (1.5,0) -- (2.25,0);
\node at (3,0) {$\cdots$};
\draw[-,thick] (3.75,0) -- (4.5,0);
\node[below] at (5.25,0) {$\xw_1$};
\draw[o-,thick] (4.5,0) -- (5.25,0);
\node[above] at (4.25,0) {$\ast$};
\node at (6,0) {$\cdots$};
\draw[-o,thick] (6.75,0) -- (7.5,0);
\node[below] at (7.5,0) {$\nmax$};
\draw[-,thick,dashed] (7.5,0) -- (8.25,0);
\node at (9,0) {$\cdots$};
\draw[-o,thick,dashed] (9.75,0) -- (10.5,0);
\node[below] at (9.75,0) {$\hw_1$};
\draw[-,thick,dashed] (10.5,0) -- (11.25,0);
\node[above] at (10.75,0) {$\ast$};
\node at (12,0) {$\cdots$};
\draw[-,thick,dashed] (12.75,0) -- (13.5,0);
\draw[o-o,thick,dashed] (13.5,0) -- (15,0);
\node[below] at (14.25,0) {$\hwp_1$};
\end{tikzpicture}
\]
\begin{remark}
We believe that the weight space $\wVbwkp\subset \xRbk$ can be represented by any \sprd\ quivers, that is, the marked and unmarked edges are distributed arbitrarily as long as the \aedg s are to the left of the \bedg s. One can also transpose two unmarked edges or two marked edges by the Hanany-Witten move~\eqref{eq:hwmvp}, however we do not know whether it is possible to transpose a marked edge and an unmarked edge of opposite nature.
\end{remark}

\subsubsection{Varieties for weight spaces}

For a Lie superalgebra $\glmn$ we consider the weight space of $(\bfw,\bfwp)$  in the module $\xRbkz = \xRv{\hw_1}\otimes\cdots\otimes\xRv{\hw_{\rM}}$. The corresponding marked \sprd\ quiver $\xQbwpk$ has the form
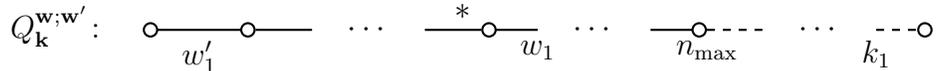
\[
\xQbwpk\colon\quad
\begin{tikzpicture}[scale=1,baseline=-0.1cm]
\draw[o-o,thick] (0,0) -- (1.5,0);
\node[below] at (0.75,0) {$\xwp_1$};
\draw[-,thick] (1.5,0) -- (2.25,0);
\node at (3,0) {$\cdots$};
\draw[-,thick] (3.75,0) -- (4.5,0);
\node[below] at (5.25,0) {$\xw_1$};
\draw[o-,thick] (4.5,0) -- (5.25,0);
\node[above] at (4.25,0) {$\ast$};
\node at (6,0) {$\cdots$};
\draw[-o,thick] (6.75,0) -- (7.5,0);
\draw[-,thick,dashed] (7.5,0) -- (8.25,0);
\node at (9,0) {$\cdots$};
\draw[-o,thick,dashed] (9.75,0) -- (10.5,0);
\node[below] at (9.75,0) {$\hw_1$};
\node[below] at (7.5,0) {$\nmax$};
\end{tikzpicture}
\]
and we denote $\xXbwpk$ the corresponding variety.
Similar to~\eqref{eq:splq}, we split this quiver into the arrow half
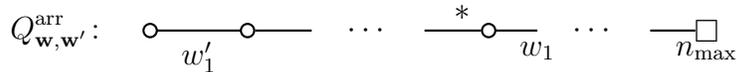
\[
\xQawp\colon\quad
\begin{tikzpicture}[scale=1,baseline=-0.1cm]
\draw[o-o,thick] (0,0) -- (1.5,0);
\node[below] at (0.75,0) {$\xwp_1$};
\draw[-,thick] (1.5,0) -- (2.25,0);
\node at (3,0) {$\cdots$};
\draw[-,thick] (3.75,0) -- (4.5,0);
\node[below] at (5.25,0) {$\xw_1$};
\draw[o-,thick] (4.5,0) -- (5.25,0);
\node[above] at (4.25,0) {$\ast$};
\node at (6,0) {$\cdots$};
\draw[-,thick] (6.75,0) -- (7.35,0);
\node at (7.5,0) {$\Box$};
\node[below] at (7.5,0) {$\nmax$};
\end{tikzpicture}
\]
and the bow half $\xQbw$. Each half-quiver produces its own \Gpr, and the variety $\xXbwpk$ of the full quiver is their \sint\ with respect to $\rmTs\glv{\nmax}$.

The bow half-quiver $\xQbw$ yields the \Gpr\ $\xGprv{ \xSSbhw}{\Tr\xmu_{\xSS}X}$, where $\xmuS$
is the moment map of the action of $\GLv{\nmax}$ on $\xSSbhw$.

%
%
%
%

The arrow half-quiver $\xQawp$ yields the \Gpr\ $\xGprv{\hflFm}{\Tr\xmuF X}$. Here $\hflFm$ is
a `mixed parabolic-nilpotent' vector bundle 
over the flag variety  $\flFwm = \{\Fb \}$ which corresponds to the concatenated weight list $(\bfwp,\bfw)$.
The fiber of $\hflF_{\bfwp,\bfw}$ over  a partial flag $\Fb$ is the subspace $\xsVFb\subset \End(\C^{\nmax})$ such that $\xphi\in \xsVFb$ if
\[
\xphi(F_{i}) \subset
\begin{cases}
F_{i},&\text{if the $i\leq \rN$,}
\\
F_{i-1},&\text{if the $i>\rN$.}
\end{cases}
\]
The function $\xmuF\colon \hflFm\rightarrow \glv{\nmax}$ is defined as $\xmuF(\Fb,\xphi) = \xphi$.

Thus the variety $\xXbwpk$ is the \sint:
\begin{equation}
\label{eq:frstex}
\begin{split}
\xXbwpk   & = \xGprv{\hflFm}{\Tr\xmuF X}\smcap\xGprv{ \xSSbhw}{\Tr\xmuS X}
\\
&= \Critvv{\Tr(\xmuF-\xmuS)X}{\hflFm\times\xSSbhw\times\glnmax}/\GLnmax.
\end{split}
\end{equation}
The criticality with respect to $\glnmax$ requires $\xmuF= \xmuS$. Since $\xmuS$ is the moment map for the action of $\GLnmax$ on $\xSSbhw$ and this action is free, it follows that the criticality of $\Tr\xmuS X$ along $\xSSbhw$ requires $X=0$. The variation of $\Tr\xmuF X$ along $\hflFm$ is proportional to $X$, so $X=0$ guarantees that this variation is zero. Hence the critical locus of~\eqref{eq:frstex} is just the condition $\xmuF=\xmuS$ imposed on $\hflFm\times\xSSbhw$,
so $\xXbwpk$ has a quiver-like presentation:
\[
\xXbwpk = (\hflFm\times\xSSbhw)\Big|_{\xmuF=\xmuS} \big/\GLnmax.
\]

If we consider the tensor product of defining representations $\xRbo=\xRo\otimes\cdots\otimes\xRo$, that is,
$\bfhw = \bfo = (1,\ldots,1)$, then $\xSSbo = \rmTs \GLnmax$ and the corresponding variety is the mixed bundle to the partial flag variety:
\begin{equation}
\label{eq:glmndef}
\xXbwpko = \rmTs\xFLbw.
\end{equation}

\subsubsection{$\glN$ presented as $\glzn$}
Finally, consider the case of $\gM=0$, that is, the algebra is $\glN$, but it is presented as $\glzn$ rather than as traditional $\glnz$. Denote $\xRN$ the (ordinary, even) defining representation of $\glN$ and consider the product of its symmetric powers
\[
\xRsbk =  \smSv{\hw_1}\xRN\otimes\cdots\otimes\smSv{\hw_{\rM}}\xRN.
\]
The defining representation $\xRo$ of $\glzn$ is odd, so its exterior powers appearing in $\xRbkz$ are, in fact, symmetric powers of $\xRN$: $\xRbkz = \xRsbk$. Hence, according to the general construction, the weight $\bfw$ subspace in the product of symmetric powers $\xRsbk$ is represented by the \sint\ of bundle of parabolic algebras over the flag variety $\xFLbw$ and the equivariant Slodowy slice
\[
\xYbwk =
(\Prbw\times\xSSbhw)\Big|_{\xmuF=\xmuS} \big/\GLnmax,
\]
where $\Prbw$ is a bundle over $\xFLbw$, whose fiber
 over  a partial flag $\Fb$ is the subspace $\xsVFb\subset \End(\C^{\nmax})$ such that $\xphi\in \xsVFb$ if
 $\xphi(F_{i})\subset F_i$ for all $i$, while $\xmuF(\Fb,\xphi) = \xphi$.

\section{The spaces $X^{(r)}_{k,n}$ and their equivariant cohomology}\label{secr}

From now on in the whole paper we will focus on the construction of Section \ref{sec24} in the special case of $M=N=1$, that is, corresponding to the decorated quivers

\begin{equation}\label{thefourquivers}
\begin{tikzpicture}[baseline=-60pt]
\draw[o-o, thick] (0,0) -- (1.5,0);
\node[below] at (0.75,0) {$k$};
\draw[-o,thick] (1.5,0) -- (3,0);
\node[below] at (2.25,0) {$n-k$};
\draw[-o,dashed, thick] (3,0) -- (4.5,0);
\node[below] at (3.75,0) {$1$};
\draw[-o,dashed, thick] (4.5,0) -- (6,0);
\node[below] at (5.25,0) {$1$};
\draw[-o,dashed, thick] (6,0) -- (7.5,0);
\node[below] at (6.75,0) {$1$};
\node at (8,0) {$\cdots$};
\draw[-o,dashed, thick] (8.5,0) -- (9.5,0);
\node[below] at (8.8,0) {$1$};
\begin{scope}[yshift=-1.5cm]
\draw[o-o, thick] (0,0) -- (1.5,0);
\node[below] at (0.75,0) {$k$};
\draw[-o,thick] (1.5,0) -- (3,0);
\node[below] at (2.25,0) {$n-k$};
\draw[-o,dashed, thick] (3,0) -- (4.5,0);
\node[below] at (3.75,0) {$1$};
\draw[-o,dashed, thick] (4.5,0) -- (6,0);
\node[below] at (5.25,0) {$1$};
\draw[-o,dashed, thick] (6,0) -- (7.5,0);
\node[below] at (6.75,0) {$1$};
\node at (8,0) {$\cdots$};
\draw[-o,dashed, thick] (8.5,0) -- (9.5,0);
\node[below] at (8.8,0) {$1$};
\node[above] at (0.75,0) {$\ast$};
\end{scope}
\begin{scope}[yshift=-3cm]
\draw[o-o, thick] (0,0) -- (1.5,0);
\node[below] at (0.75,0) {$k$};
\draw[-o,thick] (1.5,0) -- (3,0);
\node[below] at (2.25,0) {$n-k$};
\draw[-o,dashed, thick] (3,0) -- (4.5,0);
\node[below] at (3.75,0) {$1$};
\draw[-o,dashed, thick] (4.5,0) -- (6,0);
\node[below] at (5.25,0) {$1$};
\draw[-o,dashed, thick] (6,0) -- (7.5,0);
\node[below] at (6.75,0) {$1$};
\node at (8,0) {$\cdots$};
\draw[-o,dashed, thick] (8.5,0) -- (9.5,0);
\node[below] at (8.8,0) {$1$};
\node[above] at (2.25,0) {$\ast$};
\end{scope}
\begin{scope}[yshift=-4.5cm]
\draw[o-o, thick] (0,0) -- (1.5,0);
\node[below] at (0.75,0) {$k$};
\draw[-o,thick] (1.5,0) -- (3,0);
\node[below] at (2.25,0) {$n-k$};
\draw[-o,dashed, thick] (3,0) -- (4.5,0);
\node[below] at (3.75,0) {$1$};
\draw[-o,dashed, thick] (4.5,0) -- (6,0);
\node[below] at (5.25,0) {$1$};
\draw[-o,dashed, thick] (6,0) -- (7.5,0);
\node[below] at (6.75,0) {$1$};
\node at (8,0) {$\cdots$};
\draw[-o,dashed, thick] (8.5,0) -- (9.5,0);
\node[below] at (8.8,0) {$1$};
\node[above] at (0.75,0) {$\ast$};
\node[above] at (2.25,0) {$\ast$};
\end{scope}
\end{tikzpicture}.
\end{equation}
Now we give a detailed description of the corresponding varieties and their equivariant cohomology.

\subsection{The spaces $X^{(r)}_{k,n}$}
\label{sec:spaces}
Consider the tautological short exact sequence $0\to S\to \C^n\to Q \to 0$ of vector bundles over $ \Grkn$. Define
\begin{itemize}
\item $X^{\zz}_{k,n}=$total space of $\Hom(Q,S)=\rmTs\! \Grkn$;
\item $X^{\oz}_{k,n}=$total space of $\Hom(\C^n,S)$;
\item $X^{\zo}_{k,n}=$total space of $\Hom(Q,\C^n)$;
\item $X^{\oo}_{k,n}=$total space of $\Hom(S,S)\oplus\Hom(Q,\C^n)=\Hom(\C^n,S)\oplus\Hom(Q,Q)$
\end{itemize}
illustrated in Figure \ref{fig1}.
Several notions and statements below will have four versions, corresponding to these four spaces. The upper index $(r)=(00), (10), (01), (11)$ will always refer to this choice.

\subsection{Torus equivariant cohomology of $\Gr_k\C^n$}\label{sec:GKM}

The natural action of $A=A^n=(\C^\times)^n$ on $\C^n$ induces an action on $ \Grkn$. The fixed points of the action are the coordinate $k$-planes. The one naturally corresponding to the $k$-element subset $I\subset \{1,\ldots,n\}$ will be denoted by $p_I$. The set of $k$-element subsets of $\{1,\ldots,n\}$ will be denoted by $\I_k$.

We have $H^*_A(\pt)=\C[z_1,\ldots,z_n]$, where $z_i$ is the first Chern class of the tautological line bundle over $B(\C^\times)$ (the $i$th $\C^\times$ factor). The $A$ equivariant cohomology ring of any space with an $A$ action is hence a $\C[z_1,\ldots,z_n]$-module.

\begin{figure}
\begin{tikzpicture}[scale=1,baseline=0]
\draw[fill=blue!40] (1,1) to (2,1) to (2,2) to (1,2);
\draw (0,0) to (2,0) to (2,2) to (0,2) to (0,0);
\draw (1,0) to (1,2); \draw (0,1) to (2,1);
\node at (-.3,.5) {$Q$}; \node at (-.3,1.5) {$S$}; \node at (.5,2.3) {$S$}; \node at (1.5,2.3) {$Q$};
\node at (1,-.5) {$(00)$};
\end{tikzpicture}
\qquad\qquad
\begin{tikzpicture}[scale=1,baseline=0]
\draw[fill=blue!40] (0,1) to (2,1) to (2,2) to (0,2);
\draw (0,0) to (2,0) to (2,2) to (0,2) to (0,0);
\draw (1,0) to (1,2); \draw (0,1) to (2,1);
\node at (-.3,.5) {$Q$}; \node at (-.3,1.5) {$S$}; \node at (.5,2.3) {$S$}; \node at (1.5,2.3) {$Q$};
\node at (1,-.5) {$(10)$};
\end{tikzpicture}
\qquad\qquad
\begin{tikzpicture}[scale=1,baseline=0]
\draw[fill=blue!40] (1,0) to (2,0) to (2,2) to (1,2);
\draw (0,0) to (2,0) to (2,2) to (0,2) to (0,0);
\draw (1,0) to (1,2); \draw (0,1) to (2,1);
\node at (-.3,.5) {$Q$}; \node at (-.3,1.5) {$S$}; \node at (.5,2.3) {$S$}; \node at (1.5,2.3) {$Q$};
\node at (1,-.5) {$(01)$};
\end{tikzpicture}
\qquad\qquad
\begin{tikzpicture}[scale=1,baseline=0]
\draw[fill=blue!40] (1,1) to (1,0) to (2,0) to (2,2) to (0,2) to (0,1);
\draw (0,0) to (2,0) to (2,2) to (0,2) to (0,0);
\draw (1,0) to (1,2); \draw (0,1) to (2,1);
\node at (-.3,.5) {$Q$}; \node at (-.3,1.5) {$S$}; \node at (.5,2.3) {$S$}; \node at (1.5,2.3) {$Q$};
\node at (1,-.5) {$(11)$};
\end{tikzpicture}
\caption{Illustration of the bundles over $\Grkn$}\label{fig1}
\end{figure}

Let us recall the description of $H^*_A(\Grkn)$ based on the maps
\begin{equation}\label{eq:Hdesc}
\begin{tikzcd}
\C[\underbrace{t_1,\ldots,t_k}_{S_k},z_1,\ldots,z_n]^{S_k} \ar[r,"q",twoheadrightarrow] &
H^*_A(\Grkn) \ar[r,"\Loc",hook] &
\bigoplus_{I\in \I_k} \underbrace{H^*_A(p_I)}_{=\C[z_1,\ldots,z_n]}.
\end{tikzcd}
\end{equation}
 The $q$-image of the variables $t_i$ are the equivariant Chern roots of the tautological $k$-bundle $S$ over $\Grkn$. They generate $H^*_A(\Grkn)$ over $H^*_A(\pt)$, hence the map $q$ is surjective.

The map $\Loc$ is the restriction (``equivariant localization'') map in cohomology to the union of fixed points. It is injective, and its image has the so-called GKM description \cite{GKM}:

\begin{quotation}{\em
The tuple $(f_I)_{I\in \I_k}$ belongs to the image of $\Loc$ if and only if for any two components $f_I$, $f_J$ satisfying $I=K \cup \{i\}$, $J=K \cup \{j\}$ ($|K|=k-1$, $i\not=j$) we have $(z_i-z_j) |(f_I-f_J)$ in $\C[z_1,\ldots,z_n]$.}
\end{quotation}

\noindent Hence, if we allowed $z_i-z_j$ denominators, ie. by tensoring with $\C(z_1,\ldots,z_n)$, then the $\Loc$ map would become an isomorphism.

The $I$ component of the composition $\Loc \circ q$ is obtained by substituting $t_s=z_{i_s}$ for $I=\{i_1,\ldots,i_k\}$, which we will write as
\begin{equation}\label{eq:AlgLoc}
\Loc \circ q: f(t,z)\mapsto ( f(z_I,z) )_{I\in \I_k}.
\end{equation}

In summary, we have two ways of naming an element in $H^*_A(\Grkn)$. Either by an $\binom{n}{k}$ tuple of polynomials satisfying the GKM condition, or by an element of $\C[{t_1,\ldots,t_k},z_1,\ldots,z_n]^{S_k}$---although this latter element is only unique up to the kernel of \eqref{eq:AlgLoc}.

\subsection{The $\X_n^{(r)}$ spaces, and their $T$ equivariant cohomology}

We define
\[
\Xrn=\bigsqcup_{k=0}^n X^{(r)}_{k,n} \qquad\qquad \text{for $r=00, 10,01,11$}.
\]
The $A=(\C^\times)^n$ action on $\Grkn$ induces an action on  $X_{k,n}^{(r)}$, and hence on $\Xrn$. We let an extra $\C^\times$ (denoted by $\C^\times_{\h}$) act on $X_{k,n}^{(r)}$ (and hence on $\Xrn$) by multiplication in the fibers. Thus we have $\T=\T^n=A \times \C^\times_{\h}$ actions on $X_{k,n}^{(r)}$ and $\Xrn$.

The $X_{k,n}^{(r)}$ spaces are $\T$ equivariantly homotopy equivalent to $\Grkn$, and hence we have
\begin{equation}\label{eq:random}
H^*_{\T}(\Xrn) = \bigoplus_{k=0}^n H_A^*(\Grkn) \otimes \C[\h] \qquad\qquad\qquad \forall r.
\end{equation}

\subsection{The $\Loc$ map on $H^*_{\T}(\Xrn)$} \label{sec:Loc}

We can regard the $\Loc$ map as a map
\[
H^*_{\T}(\Xrn) \to \bigoplus_{I\subset \{1,\ldots,n\}} \C[z_1,\ldots,z_n,\h].
\]
It will be convenient for us to permit rational function coefficients: define $\Hb_n=H^*_{\T}(\Xrn) \otimes \C(z_1,\ldots,z_n,\h)$---we dropped the upper index $r$ because of the independence on $r$, see \eqref{eq:random}. This way we can regard $\Loc$, which is now an isomorphism of $2^n$-dimensional vector spaces over $\C(z_1,\ldots,z_n,\h)$, as
\begin{equation}\label{eq:Loc1}
\begin{tikzcd}
\Hb_n \ar[rr,"\Loc"] & & \bigoplus_{I\subset \{1,\ldots,n\}} \C(z_1,\ldots,z_n,\h). 
\end{tikzcd}
\end{equation}

In Section \ref{sec:stab} we will consider four versions of $n!$ different isomorphisms from right to left in~\eqref{eq:Loc1}: the super stable envelope maps.

\subsection{Tangent weights at torus fixed points}

The tangent space of $X_{k,n}^{(r)}$ at the torus fixed point $p_I$, as a $\T$ representation, will be denoted by $T_I^{(r)}$. It splits to ``horizontal'' and ``vertical'' sub-representations
\[
T^{(r)}_I=T^{(r),hor}_I \oplus T^{(r), ver}_I
\]
where $T^{(r),hor}_I$ is the tangent space of $ \Grkn$ at $p_I$, and $T^{(r), ver}_I$ is the vector bundle defined in Section \ref{sec:spaces} restricted to $p_I$. The weights of $T^{(r),hor}_I$ (called horizontal weights) are $z_j-z_i$ for $i\in I, j\in \bar{I}$. The weights of $T^{(r), ver}_I$, called vertical weights, can be read from Figure \ref{fig1}:

\begin{tabular}{lll}
(r=00) & $z_i-z_j+\h$ & for $i\in I, j\in \bar{I}$,   \\
(r=10) & $z_i-z_s+\h$ & for $i\in I, s\in \{1,\ldots,n\}$, \\
(r=01) & $z_s-z_j+\h$ & for $j\in \bar{I}, s\in \{1,\ldots,n\}$, \\
(r=11) & $z_i-z_j+\h$ &for $i, j\in I$ and \\
     & $z_i-z_j+\h$ &for $i,j \in \bar{I}$ and \\
     & $z_i-z_j+\h$ &for $i\in I, j\in \bar{I}$.
\end{tabular}

\subsection{Repelling and attracting directions}
Given a permutation $\sigma\in S_n$ we call a weight $z_i-z_j+\epsilon \h$ (where $\epsilon=\{0,1\}$)

\begin{tabular}{ll}
{\em $\sigma$-repelling} & if  $\sigma^{-1}(i)>\sigma^{-1}(j)$, \\
{\em $\sigma$-attracting} & if $\sigma^{-1}(i)<\sigma^{-1}(j)$,  \\
{\em $\sigma$-neutral} & if $\sigma^{-1}(i)=\sigma^{-1}(j)$.
\end{tabular}

\noindent In notation we will use the signs $-$, $+$, $0$ referring to repelling, attracting, neutral weights. For fixed $\sigma$ we have the further splitting
\[
T_I^{(r)} =
\underbrace{\left(
T^{(r),hor,\sigma +}_I \oplus T^{(r),hor,\sigma -}_I
\right)}_{T_I^{(r),hor}}
\bigoplus
\underbrace{
\left(
T^{(r),ver,\sigma+}_I \oplus T^{(r),ver,\sigma-}_I \oplus T^{(r),ver,\sigma 0}_I
\right)
}_{T_I^{(r),ver}}
\]
according to $\sigma$-attracting/repelling/neutral directions. The $\T$-equivariant Euler class of these representations will be decorated by indexes the same way. For example we have
\[
e^{(r),hor,\sigma-}_I
=
e(T_I^{(r),hor,\sigma-})
=
\mathop{\prod_{i\in I, j\in \bar{I}}}_{\sigma^{-1}(j)>\sigma^{-1}(i)} (z_j-z_i)
\]
for any $r$, or
\[
e^{(10),ver,\sigma-}_I
=
e(T_I^{(10),ver,\sigma-})
=
\mathop{\prod_{i\in I, s\in \{1,\ldots,n\}}}_{\sigma^{-1}(j)>\sigma^{-1}(s)} (z_i-z_s+\h).
\]

The dimension of the space $T_I^{(r),hor,\sigma -}\oplus T_I^{(r),ver,\sigma -}$ does not depend on $I$ or on $\sigma$; it only depends on $r$. Let us denote this dimension by $d^{(r)}$. That is (cf. Figure \ref{fig2}),
\begin{multline*}
d^{\zz}=k(n-k), \qquad
d^{\oz}=k(n-k)+\binom{k}{2}, \\
d^{\zo}=k(n-k)+\binom{n-k}{2}, \qquad
d^{\oo}=k(n-k)+\binom{k}{2}+\binom{n-k}{2}=\binom{n}{2}.
\end{multline*}

\begin{remark}
The appearance of neutral weights for $r=(10), (01), (11)$ is a novelty. In the language of Section \ref{sec:string} it is due to the fact that a \Dx\ brane and a \NSx\ brane share a common direction $\IRtwx$, so a D2 brane sandwiched between them can move along $\IRtwx$.
\end{remark}

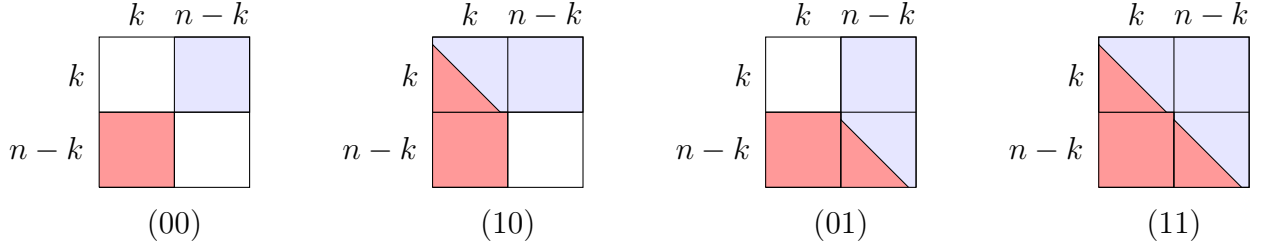
\begin{figure}
\begin{tikzpicture}[scale=1,baseline=0]
\draw[fill=blue!10] (1,1) to (2,1) to (2,2) to (1,2);
\draw[fill=red!40] (0,0) to (1,0) to (1,1) to (0,1);
\draw (0,0) to (2,0) to (2,2) to (0,2) to (0,0);
\draw (1,0) to (1,2); \draw (0,1) to (2,1);
\node at (-.7,.5) {$n-k$}; \node at (-.3,1.5) {$k$}; \node at (.5,2.3) {$k$}; \node at (1.5,2.3) {$n-k$};
\node at (1,-.5) {$(00)$};
\end{tikzpicture}
\qquad
\begin{tikzpicture}[scale=1,baseline=0]
\draw[fill=blue!10] (0,1) to (2,1) to (2,2) to (0,2);
\draw[fill=red!40] (0,0) to (1,0) to (1,1) to (0,1);
\draw[fill=red!40] (0.9,1) to (0,1.9) to (0,1);
\draw (0,0) to (2,0) to (2,2) to (0,2) to (0,0);
\draw (1,0) to (1,2); \draw (0,1) to (2,1);
\node at (-.7,.5) {$n-k$}; \node at (-.3,1.5) {$k$}; \node at (.5,2.3) {$k$}; \node at (1.5,2.3) {$n-k$};
\node at (1,-.5) {$(10)$};
\end{tikzpicture}
\qquad
\begin{tikzpicture}[scale=1,baseline=0]
\draw[fill=blue!10] (2,0) to (2,2) to (1,2) to (1,0);
\draw[fill=red!40] (0,0) to (1,0) to (1,1) to (0,1);
\draw[fill=red!40] (1.9,0) to (1,0.9) to (1,0);
\draw (0,0) to (2,0) to (2,2) to (0,2) to (0,0);
\draw (1,0) to (1,2); \draw (0,1) to (2,1);
\node at (-.7,.5) {$n-k$}; \node at (-.3,1.5) {$k$}; \node at (.5,2.3) {$k$}; \node at (1.5,2.3) {$n-k$};
\node at (1,-.5) {$(01)$};
\end{tikzpicture}
\qquad
\begin{tikzpicture}[scale=1,baseline=0]
\draw[fill=blue!10] (1,1) to (1,0) to (2,0) to (2,2) to (0,2) to (0,1);
\draw[fill=red!40] (0,0) to (1,0) to (1,1) to (0,1);
\draw[fill=red!40] (1.9,0) to (1,0.9) to (1,0);
\draw[fill=red!40] (0.9,1) to (0,1.9) to (0,1);
\draw (0,0) to (2,0) to (2,2) to (0,2) to (0,0);
\draw (1,0) to (1,2); \draw (0,1) to (2,1);
\node at (-.7,.5) {$n-k$}; \node at (-.3,1.5) {$k$}; \node at (.5,2.3) {$k$}; \node at (1.5,2.3) {$n-k$};
\node at (1,-.5) {$(11)$};
\end{tikzpicture}
\caption{Red regions illustrate the dimensions $d^{(r)}$.}\label{fig2}
\end{figure}

\section{Super stable envelopes}\label{sec:stab}

\subsection{Definition}
Let us fix $k,n$ and $\sigma\in S_n$. The map
\[
\begin{tikzcd}
\Stab^{(r)}_{\sigma}: & H^*_{\T}\left((X^{(r)}_{k,n})^{\T}\right) \ar[r] \ar[equals]{d} & H_{\T}^*(X_{k,n}^{(r)}) \ar[equals]{d} \\
& \bigoplus_{I\in \I_k} \C[z,\h] & H^*_{\T}( \Grkn)
\end{tikzcd}
\]
is called the $\sigma$ super stable envelope (map), if the classes $\kappa^{(r)}_{\sigma,I}=\Stab_\sigma^{(r)}\left(1_I\right)$ satisfy the axioms
\begin{description}
\item[A0] $\deg (\kappa^{(r)}_{\sigma,I} ) = d^{(r)}$;
\item[A1] $\kappa^{(r)}_{\sigma,I}|_I= e_I^{(r),ver,\sigma-} e_I^{(r),hor,\sigma-}$;
\item[A2] $\kappa^{(r)}_{\sigma,I}|_J$ is divisible by $\h$ for $J\not= I$;
\item[A3] $\kappa^{(r)}_{\sigma,I}|_J$ is divisible by $e_J^{(r),ver,\sigma-}$ for all $J$.
\end{description}

\noindent In the {\bf A0} axiom we mean that the class is of homogeneous degree $d^{(r)}$ where $\deg z_i=\deg \h=\deg t_i=1$ (that is, degree $d$ classes live in $H^{2d}$).

The $\Stab^{\zz}_\sigma$ maps coincide with the stable envelope maps of Maulik-Okounkov \cite{MO} for the quiver variety $\cup_k \rmTs\!\Gr_k\C^n$.

If the $\Stab_{\sigma}^{(r)}$ maps exist then they are uniquely determined by the axioms. The proof of this statement is the same as the proof of the existence of stable envelopes in the known cases in the literature \cite[Section 3.3.4]{MO} (c.f. \cite[Section 3.1]{RTVtrig}, \cite[Section 7.8]{RTVEllK}). We will prove the existence of stable envelopes in Section \ref{sec:existence}.

Now we give examples for $\kappa^{(r)}_{\sigma,I}$ classes. It is instructive to verify the axioms for these examples.

\subsection{Example: $\PPP^1$}\label{ex:P1}
 Let $n=2, k=1, S_2=\{\id, s\}$. The classes $\kappa^{(r)}_{\sigma,\{i\}}$ are elements of $H^*_{\T}(X_{1,2}^{(r)})=H^*_{\T}(\PPP^1)$. We have two ways of naming such elements, see Section \ref{sec:GKM}, either by a GKM-consistent pair of polynomials in $\C[z_1,z_2]$, or by a representative in $\C[t_1,z_1,z_2]$.  Accordingly, we have
\[
\begin{tabular}{cclcrcl}
$\kappa^{(r)}_{\id,\{1\}}$ & $=$ & $(z_2-z_1$ & $,$ & $0)$ & $=$ & $[z_2-t_1]$, \\
$\kappa^{(r)}_{\id,\{1\}}$ & $=$ & $(\h$ & $,$ & $z_2-z_1+\h)$ & $=$ & $[t_1-z_1+\h]$, \\
$\kappa^{(r)}_{s,\{1\}}$ & $=$ & $(z_1-z_2+\h$ & $,$ & $\h)$ & $=$ & $[t_1-z_2+\h]$, \\
$\kappa^{(r)}_{s,\{1\}}$ & $=$ & $(0$ & $,$ & $z_1-z_2)$ & $=$ & $[z_1-t_1]$
\end{tabular}
\]
for all $r=00,10,01,11$.

\subsection{Example: projective spaces}
For $r=00,10$ the example of Section \ref{ex:P1} generalizes to $k=1$ and arbitrary $n$. Namely, the polynomial
\[
\prod_{b=1}^{i-1}(t_1-z_b+\h) \prod_{b=i+1}^n (z_b-t_1)
\]
represents the classes $\kappa^{\zz}_{\id,\{i\}}=\kappa^{\oz}_{\id,\{i\}}$.

The $k=1$ (i.e. $\PPP^{n-1}$) formulas for $r=01,11$ are less obvious. While the polynomial
\[
\prod_{b=2}^n(z_b-t_1) \prod_{2\leq a<b\leq n} (z_b-z_a+\h)
\]
represents $\kappa_{\id,\{1\}}^{\zo}=\kappa_{\id,\{1\}}^{\oo}$ for any $n$,  no such ``nicely factoring'' polynomial representative exists in general. For example, for $n=4$ the `best' polynomial representative for
\begin{multline}\label{thatone}
\kappa_{\id,\{2\}}^{\zo}=\kappa_{\id,\{2\}}^{\oo}=
\Big(
(z_3-z_1)(z_4-z_1)(z_3-z_2+\h)(z_4-z_2+\h)(z_4-z_3),
\\
(z_3-z_2)(z_4-z_2)(z_2-z_1+\h)(z_3-z_1+\h)(z_4-z_1+\h)(z_4-z_3+\h),0,0\Big)
\end{multline}
we found is
\begin{multline*}
(t-z_1+\h)(z_3-t)(z_4-t)(z_4-z_3+\h)\times \\
(-t^2+t(z_3+z_4+2\h)+\h^2+\h(-2z_1-2z_2+z_3+z_4)+z_1^2+z_2^2+z_3z_4-(z_1+z_2)(z_3+z_4)).
\end{multline*}
For general $r,k,n$ neither the fixed point restrictions nor the polynomial representatives are products of linear factors.  In Sections~\ref{sec:weight}--\ref{sec:existence} we will use a further algebraic trick to name the $\kappa^{(r)}_{\sigma,I}$ classes.

\begin{remark}
In our description of $H^*_{\T}$ of Grassmannians we permitted the Chern roots $t_1,\ldots,t_k$ of the tautological bundles. If we included the Chern roots, say, $t'_1,\ldots,t'_{n-k}$ of the quotient bundle as well, we would have more freedom to name polynomial representatives of $\kappa$ classes. However, that approach has disadvantages when considering quivers instead of Grassmannians, so we do not pursue it.
\end{remark}

\section{Super weight functions}\label{sec:weight}

In this section we introduce four versions of rational functions in the variables
\begin{equation}\label{vars}
\begin{tabular}{ll}
$t_1,t_2,\ldots,t_k$ & (``Chern root variables''), \\
$z_1,z_2,\ldots\ldots,z_n$ & (``equivariant variables''),
\end{tabular}
\end{equation}
that will---in an implicit way---provide formulas for the super stable envelopes of Section \ref{sec:stab}. The $r=00$ version is (up to convention changes) the Tarasov-Varchenko weight function (\cite{TV,RTVpartial}), the other ones are superalgebra generalizations of it.

As before, $k\leq n$ are non-negative integers, and the set of $k$-element subsets of $\{1,2,\ldots,n\}$ is denoted by $\I_k$. For $I\in \I_k$ we will use the notation $I=\{i_1<i_2<\ldots<i_k\}$, and we define
\[
\Sym_k f(t_1,\ldots,t_k)=\sum_{\tau\in S_k} (\tau f)=\sum_{\tau\in S_k} f(t_{\tau(1)},\ldots,t_{\tau(k)}).
\]

\subsection{Version $r=00$} (The classical nilpotent version.) Consider the rational function
$
W^{\zz}_I=\Sym_k (U^{\zz}_I)
$
where
\[
U^{\zz}_I=
\prod_{a=1}^k \left(
\prod_{b=1}^{i_a-1} (t_a-z_b+\h)
\prod_{b=i_a+1}^n (z_b-t_a)\right)
\cdot
\prod_{a=1}^k\prod_{b=a+1}^k \frac{1}{(t_b-t_a+\h)(t_b-t_a)}.
\]
For a permutation $\sigma\in S_n$ define the (cohomological) {\em $r=00$ super weight function}
\[
W^{\zz}_{\sigma,I}
=
W^{\zz}_{\sigma^{-1}(I)}(t_1,\ldots,t_k,z_{\sigma(1)},\ldots,z_{\sigma(n)}).
\]

\begin{example} \rm Let $n=2$ and $S_2=\{\id,s\}$. We have
\[
\def\arraystretch{1.6}
\begin{tabular}{ll}
$W^{\zz}_{\id, \{\}} =1$ & $W^{\zz}_{s,\{\}}=1$ \\
$W^{\zz}_{\id, \{1\}} = z_2-t_1$ & $W^{\zz}_{s,\{1\}}=t_1-z_2+\h$ \\
$W^{\zz}_{\id, \{2\}} = t_1-z_1+\h$ & $W^{\zz}_{s,\{2\}}=z_1-t_1$ \\
$W^{\zz}_{\id, \{1,2\}} = \Sym_2 \dfrac{(t_2-z_1+\h)(z_2-t_1)}{(t_2-t_1+\h)(t_2-t_1)}$
& $W^{\zz}_{s,\{1,2\}}= \Sym_2 \dfrac{(t_2-z_2+\h)(z_1-t_1)}{(t_2-t_1+\h)(t_2-t_1)}$.
\end{tabular}
\]
We invite the reader to verify that
\[
W^{\zz}_{\id,\{1,2\}}|_{t_1=z_1,t_2=z_2}=W^{\zz}_{s,\{1,2\}}|_{t_1=z_1,t_2=z_2}=1,
\]
and that
\[
\renewcommand\arraystretch{1.6}
\left[\begin{array}{c}
W^{\zz}_{s,\{\}} \\
W^{\zz}_{s,\{1\}} \\
W^{\zz}_{s,\{2\}} \\
W^{\zz}_{s,\{1,2\}}
\end{array}\right]
=
\left[\begin{array}{c|cc|c}
1 & 0 & 0 & 0 \\ \hline
0 & \frac{z_1-z_2}{z_2-z_1+\h} & \frac{\h}{z_2-z_1+\h} & 0 \\
0 & \frac{\h}{z_2-z_1+\h} & \frac{z_1-z_2}{z_2-z_1+\h} & 0 \\\hline
0 & 0 & 0 & 1
\end{array}\right]
\left[\begin{array}{c}
W^{\zz}_{\id,\{\}} \\
W^{\zz}_{\id,\{1\}} \\
W^{\zz}_{\id,\{2\}} \\
W^{\zz}_{\id,\{1,2\}}
\end{array}\right].
\]
\end{example}

\subsection{Version $r=10$} Consider the rational function\footnote{in fact this one is a polynomial, c.f. the proof of Proposition \ref{prop:poly}.}
$
W^{\oz}_I=\Sym_k (U^{\oz}_I)
$
where
\[
U^{\oz}_I=
\prod_{a=1}^k \left(
\prod_{b=1}^{i_a-1} (t_a-z_b+\h)
\prod_{b=i_a+1}^n (z_b-t_a)\right)
\cdot
\prod_{a=1}^k\prod_{b=a+1}^k \frac{1}{(t_b-t_a)}.
\]
For a permutation $\sigma\in S_n$ define the (cohomological) {\em $r=10$ super weight function}
\[
W^{\oz}_{\sigma,I}
=
W^{\oz}_{\sigma^{-1}(I)}(t_1,\ldots,t_k,z_{\sigma(1)},\ldots,z_{\sigma(n)}).
\]

\begin{example} \rm
Let $n=2$ and $S_2=\{\id,s\}$. We have
\[
\renewcommand\arraystretch{1.6}
\begin{tabular}{ll}
$W^{\oz}_{\id, \{\}} =1$ & $W^{\oz}_{s,\{\}}=1$ \\
$W^{\oz}_{\id, \{1\}} = z_2-t_1$ & $W^{\oz}_{s,\{1\}}=t_1-z_2+\h$ \\
$W^{\oz}_{\id, \{2\}} = t_1-z_1+\h$ & $W^{\oz}_{s,\{2\}}=z_1-t_1$ \\
$W^{\oz}_{\id, \{1,2\}} = \Sym_2 \dfrac{(t_2-z_1+\h)(z_2-t_1)}{(t_2-t_1)}$
& $W^{\oz}_{s,\{1,2\}}= \Sym_2 \dfrac{(t_2-z_2+\h)(z_1-t_1)}{(t_2-t_1)}$ \\
\ \hskip 1.4 true cm $=z_2-z_1+\h$ & \ \hskip 1.28 true cm $=z_1-z_2+\h$.
\end{tabular}
\]
We invite the reader to verify that
\[
\renewcommand\arraystretch{1.6}
\left[\begin{array}{c}
W^{\oz}_{s,\{\}} \\
W^{\oz}_{s,\{1\}} \\
W^{\oz}_{s,\{2\}} \\
W^{\oz}_{s,\{1,2\}}
\end{array}\right]
=\left[\begin{array}{c|cc|c}
1 & 0 & 0 & 0 \\ \hline
0 & \frac{z_1-z_2}{z_2-z_1+\h} & \frac{\h}{z_2-z_1+\h} & 0 \\
0 & \frac{\h}{z_2-z_1+\h} & \frac{z_1-z_2}{z_2-z_1+\h} & 0 \\\hline
0 & 0 & 0 & \frac{z_1-z_2+\h}{z_2-z_1+\h}
\end{array}\right]
\left[\begin{array}{c}
W^{\oz}_{\id,\{\}} \\
W^{\oz}_{\id,\{1\}} \\
W^{\oz}_{\id,\{2\}} \\
W^{\oz}_{\id,\{1,2\}}
\end{array}\right].
\]
\end{example}

\subsection{Version $r=01$} Consider the rational function
$
W^{\zo}_I=\Sym_k (U^{\zo}_I)
$
where
\begin{multline*}
U^{\zo}_I=
\prod_{a=1}^k \left(
\prod_{b=1}^{i_a-1} (z_b-t_a+\h)
\prod_{b=i_a+1}^n (z_b-t_a)\right)
\cdot
\prod_{a=1}^k\prod_{b=a+1}^k \frac{1}{(t_b-t_a)}
\times\\
\h^k\prod_{a=1}^n\prod_{b=a+1}^n(z_b-z_a+\h) \prod_{a=1}^k\prod_{b=1}^n \frac{1}{z_b-t_a+\h}.
\end{multline*}
For a permutation $\sigma\in S_n$ define the (cohomological) {\em $r=01$ super weight function}
\[
W^{\zo}_{\sigma,I}
=
W^{\zo}_{\sigma^{-1}(I)}(t_1,\ldots,t_k,z_{\sigma(1)},\ldots,z_{\sigma(n)}).
\]

\begin{example} \rm
Let $n=2$ and $S_2=\{\id,s\}$. We have
\[
\renewcommand\arraystretch{1.6}
\begin{tabular}{ll}
$W^{\zo}_{\id, \{\}} =z_2-z_1+\h$ & $W^{\zo}_{s,\{\}}=z_1-z_2+\h$ \\
$W^{\zo}_{\id, \{1\}} =\dfrac{\h(z_2-t_1)(z_2-z_1+\h)}{(z_1-t_1+\h)(z_2-t_1+\h)}$ &
  $W^{\zo}_{s,\{1\}}=\dfrac{\h(z_1-z_2+\h)}{z_1-t_1+\h}$ \\
$W^{\zo}_{\id, \{2\}} = \dfrac{\h(z_2-z_1+\h)}{(z_2-t_1+\h)}$ &
$W^{\zo}_{s,\{2\}}=\dfrac{\h(z_1-t_1)(z_1-z_2+\h)}{(z_1-t_1+\h)(z_2-t_1+\h)}$ \\
$W^{\zo}_{\id, \{1,2\}} = \dfrac{\h^2(z_2-z_1+\h)(z_1-z_2+\h)}{\prod_{i=1}^2\prod_{j=1}^2 (z_i-t_j+\h)}$  & $W^{\zo}_{s,\{1,2\}}= \dfrac{\h^2(z_2-z_1+\h)(z_1-z_2+\h)}{\prod_{i=1}^2\prod_{j=1}^2 (z_i-t_j+\h)}$.
\end{tabular}
\]
We invite the reader to verify that
\[
W^{\zo}_{\id,\{1,2\}}|_{t_1=z_1,t_2=z_2}=W^{\zo}_{s,\{1,2\}}|_{t_1=z_1,t_2=z_2}=1,
\]
and that
\[
\renewcommand\arraystretch{1.6}
\left[\begin{array}{c}
W^{\zo}_{s,\{\}} \\
W^{\zo}_{s,\{1\}} \\
W^{\zo}_{s,\{2\}} \\
W^{\zo}_{s,\{1,2\}}
\end{array}\right]
=\left[\begin{array}{c|cc|c}
 \frac{z_1-z_2+\h}{z_2-z_1+\h} & 0 & 0 & 0 \\ \hline
0 & \frac{z_1-z_2}{z_2-z_1+\h} & \frac{\h}{z_2-z_1+\h} & 0 \\
0 & \frac{\h}{z_2-z_1+\h} & \frac{z_1-z_2}{z_2-z_1+\h} & 0 \\\hline
0 & 0 & 0 & 1
\end{array}\right]
\left[\begin{array}{c}
W^{\zo}_{\id,\{\}} \\
W^{\zo}_{\id,\{1\}} \\
W^{\zo}_{\id,\{2\}} \\
W^{\zo}_{\id,\{1,2\}}
\end{array}\right].
\]
\end{example}

\subsection{Version $r=11$} Consider the rational function
$
W^{\oo}_I=\Sym_k (U^{\oo}_I)
$
where
\begin{multline*}
U^{\oo}_I=
\prod_{a=1}^k \left(
\prod_{b=1}^{i_a-1} (z_b-t_a+\h)
\prod_{b=i_a+1}^n (z_b-t_a)\right)
\cdot
\prod_{a=1}^k\prod_{b=a+1}^k \frac{(t_b-t_a+\h)}{(t_b-t_a)}
\times\\
h^k
\prod_{a=1}^n\prod_{b=a+1}^n(z_b-z_a+\h) \prod_{a=1}^k\prod_{b=1}^n \frac{1}{z_b-t_a+\h}.
\end{multline*}
For a permutation $\sigma\in S_n$ define the (cohomological) {\em $r=11$ super weight function}
\[
W^{\oo}_{\sigma,I}
=
W^{\oo}_{\sigma^{-1}(I)}(t_1,\ldots,t_k,z_{\sigma(1)},\ldots,z_{\sigma(n)}).
\]

\begin{example} \rm
Let $n=2$ and $S_2=\{\id,s\}$. We have
\[
\renewcommand\arraystretch{1.6}
\begin{tabular}{ll}
$W^{\oo}_{\id, \{\}} =z_2-z_1+\h$ & $W^{\oo}_{s,\{\}}=z_1-z_2+\h$ \\
$W^{\oo}_{\id, \{1\}} =\dfrac{\h(z_2-t_1)(z_2-z_1+\h)}{(z_1-t_1+\h)(z_2-t_1+\h)}$ &
  $W^{\oo}_{s,\{1\}}=\dfrac{\h(z_1-z_2+\h)}{z_1-t_1+\h}$ \\
$W^{\oo}_{\id, \{2\}} = \dfrac{\h(z_2-z_1+\h)}{(z_2-t_1+\h)}$ &
$W^{\oo}_{s,\{2\}}=\dfrac{\h(z_1-t_1)(z_1-z_2+\h)}{(z_1-t_1+\h)(z_2-t_1+\h)}$ \\
$W^{\oo}_{\id, \{1,2\}} = \Sym_2 \frac{\h^2(z_2-z_1+\h)(t_2-t_1+\h)(z_2-t_1)}{(t_2-t_1)(z_1-t_1+\h)(z_2-t_1+\h)(z_2-t_2+\h)}$ \\
&  $W^{\oo}_{s,\{1,2\}}= \Sym_2 \frac{\h^2(z_1-z_2+\h)(t_2-t_1+\h)(z_1-t_1)}{(t_2-t_1)(z_1-t_1+\h)(z_1-t_2+\h)(z_2-t_1+\h)}$.
\end{tabular}
\]
We invite the reader to verify that
\[
W^{\oo}_{\id,\{1,2\}}|_{t_1=z_1,t_2=z_2}=z_2-z_1+\h,
\qquad
W^{\oo}_{s,\{1,2\}}|_{t_1=z_1,t_2=z_2}=z_1-z_2+\h,
\]
and that
\[
\renewcommand\arraystretch{1.6}
\left[\begin{array}{c}
W^{\oo}_{s,\{\}} \\
W^{\oo}_{s,\{1\}} \\
W^{\oo}_{s,\{2\}} \\
W^{\oo}_{s,\{1,2\}}
\end{array}\right]
=\left[\begin{array}{c|cc|c}
 \frac{z_1-z_2+\h}{z_2-z_1+\h} & 0 & 0 & 0 \\ \hline
0 & \frac{z_1-z_2}{z_2-z_1+\h} & \frac{\h}{z_2-z_1+\h} & 0 \\
0 & \frac{\h}{z_2-z_1+\h} & \frac{z_1-z_2}{z_2-z_1+\h} & 0 \\\hline
0 & 0 & 0 &  \frac{z_1-z_2+\h}{z_2-z_1+\h}
\end{array}\right]
\left[\begin{array}{c}
W^{\oo}_{\id,\{\}} \\
W^{\oo}_{\id,\{1\}} \\
W^{\oo}_{\id,\{2\}} \\
W^{\oo}_{\id,\{1,2\}}
\end{array}\right].
\]
\end{example}

\section{Properties of super weight functions}\label{sec:properties}

In this section we will show interpolation and recursion (so-called R-matrix-) properties of super weight functions.

\subsection{Interpolation properties}

The function $W^{(r)}_{\sigma,I}$ is a rational function in $t_1,\ldots,t_k$, $z_1,\ldots,z_n$, $\h$, of homogeneous degree $d^{(r)}$, symmetric in the $t_i$ variables. For $I,J\in \I_k$, we will write $W^{(r)}_{\sigma,J}(z_I,z,\h)$ for the---{\em a priori} rational function---obtained by substituting $t_s=z_{i_s}$ for $I=\{i_1,\ldots,i_k\}$ in $W^{(r)}_{\sigma,J}$, cf. \eqref{eq:AlgLoc}.

\begin{proposition}\label{prop:poly}
The function $W^{(r)}_{\sigma,I}(z_J,z,\h)$ is a {\em polynomial} in $z_1,\ldots,z_n,\h$, for all $I, J\in \I_k$.
\end{proposition}

\begin{proof}
The denominator of the $U^{(r)}_J$ function is $\prod_{a<b}^k (t_b-t_a)$ times possibly (depending on $r$) some factors of the form $(z_i-t_j+\h)$ and $(t_i-t_j+\h)$. After symmetrization $W^{(r)}_J=\Sym_k(U^{(r)}_J)$ the $(t_b-t_a)$ factors cancel, because the numerator will have poles at $t_b=t_a$ as well. The $(z_i-t_j+\h)$ and $(t_i-t_j+\h)$ factors may not cancel in $W^{(r)}_J$.

What we need to show is that {\em after the substitution} $t_s=z_{i_s}$, the resulting $(z_i-z_j+\h)$ factors cancel in the sum. This holds, because of the structure of the numerators. It is easily verified in each of the $r=00, 10, 01, 11$ cases, that a term $(\tau U^{(r)}_J)(z_I,z,\h)$ (for $\tau\in S_n$) either vanishes, or is a polynomial, i.e. the $(z_i-t_j+\h)$ factors appearing in the denominator also appear in the numerator.
\end{proof}

Now we make three propositions about the $W^{(r)}_{\sigma,J}(z_I,z,\h)$ substitutions. Each one of the three is a ``soft theorem'' in the sense that they hold {\em termwise} for $W^{(r)}_{J}(z_I,z,\h)=\sum_{\tau\in S_n} (\tau U^{(r)}_J)(z_I,z,\h)$. That is, the combinatorics of the definition of $U^{(r)}_I$ imply the statements, not the sophisticated addition of $n!$ rational functions.

\begin{proposition}\label{prop:principal}
We have
\begin{equation}\label{eq:princ}
W^{(r)}_{\sigma,I}(z_I,z,\h)=
\mathop{\prod_{1\leq a<b\leq n}}_{\clubsuit} (z_{\sigma(b)}-z_{\sigma(a)})
\cdot
\mathop{\prod_{1\leq b<a\leq n}}_{\spadesuit} (z_{\sigma(a)}-z_{\sigma(b)}+\h)
\end{equation}
where
\[
\clubsuit= \sigma(a)\in I,\sigma(b)\in \bar{I}
\]
\begin{equation}\label{eq:spade}
\spadesuit=
\begin{cases}
\sigma(a) \in I, \sigma(b)\in \bar{I} & \text{for } r=00, \\
\sigma(a) \in I & \text{for } r=10, \\
\sigma(b) \in \bar{I} & \text{for } r=01, \\
\sigma(a) \in I \text{ or } (\sigma(a)\in \bar{I} \text{ and } \sigma(b)\in \bar{I}) & \text{for } r=11.\\
\end{cases}
\end{equation}
\end{proposition}

\begin{proof}
It is enough to prove the statement for $\sigma=\id$. It follows by inspection that at the substitution $t_s=z_{i_s}$ into $W^{(r)}_{I}=\sum_{\tau\in S_k} (\tau U^{(r)}_I)$ exactly one term is not 0, the term corresponding to $\tau=\id$. Substituting $t_s=z_{i_s}$ into the $\tau=\id$ term we obtain the right hand side of \eqref{eq:princ}.
\end{proof}

Observe that on the right hand side of \eqref{eq:princ} the first product  equals $e_I^{(r),ver,\sigma-}$, and the second product equals $e_I^{(r),hor,\sigma-}$.

\begin{proposition}\label{prop:h-div}
For $J\not=I$ the polynomial $W^{(r)}_{\sigma,J}(z_I,z,\h)$ is divisible by $\h$ in $\C[z_1,$ $\ldots,$ $z_n,\h]$.
\end{proposition}

\begin{proposition}\label{prop:support}
The polynomial $W^{(r)}_{\sigma,J}(z_I,z,\h)$ is divisible by
$
\mathop{\prod_{1\leq b<a\leq n}}_{\spadesuit} (z_{\sigma(a)}-z_{\sigma(b)}+\h)
$
where $\spadesuit$ is as in \eqref{eq:spade}.
\end{proposition}

\begin{proof} (Propositions  \ref{prop:h-div}, \ref{prop:support}.)
The $\sigma=\id$ special case implies the general case. For $\sigma=\id$ the proof continues the arguments given in the proof of Proposition \ref{prop:poly}.
There we claimed that in each of the terms $(\tau U^{(r)}_J)(z_I,z,\h)$ the denominator divides the numerator, hence is a polynomial. The combinatorial structure of the numerator also implies that each of these polynomials are divisible by $\h\cdot \mathop{\prod_{1\leq b<a\leq n}}_{\spadesuit} (z_{\sigma(a)}-z_{\sigma(b)}+\h)$.
\end{proof}

\subsection{R-matrix properties}

\noindent Let $s_{a,b}$ denote the transposition in $S_n$ switching $a$ with $b$. For $\sigma,\omega\in S_n$ the permutation $\sigma\omega$ means first applying $\omega$ then applying $\sigma$. For example, the permutation $\sigma s_{a,a+1}$ is obtained from $\sigma$ by switching the $\sigma(a)$ and $\sigma(a+1)$ values.
For $I\in \I_k$, $s_{u,v}\in S_n$ we define the set $s_{u,v}(I)\in \I_k$ to be obtained from $I$ by switching $u$ and $v$. In particular, if $u,v\in I$ or if $u,v\not\in I$ then $s_{u,v}(I)=I$.

\begin{theorem}\label{thm:generalR}
Let $k\leq n$, $\sigma\in S_n$, $I\in \I_k$, $a=1,\ldots,n-1$.
\begin{itemize}
\item If ($\sigma(a) \in I,\sigma(a+1)\not\in I$) or ($\sigma(a) \not\in I,\sigma(a+1)\in I$) then
\[
W^{(r)}_{\sigma s_{a,a+1},I}=
\frac{z_{\sigma(a)}-z_{\sigma(a+1)}}{z_{\sigma(a+1)}-z_{\sigma(a)}+\h}
W^{(r)}_{\sigma,I} +
\frac{\h}{z_{\sigma(a+1)}-z_{\sigma(a)}+\h}
W^{(r)}_{\sigma,s_{\sigma(a),\sigma(a+1)}(I)}
\]
for  $r=00,10,01,11$.
\item  If $\sigma(a),\sigma(a+1)\in I$ then
\[
W^{(r)}_{\sigma s_{a,a+1},I}=
\begin{cases}
\phantom{\frac{z_{\sigma(a)}-z_{\sigma(a+1)}+\h}{z_{\sigma(a+1)}-z_{\sigma(a)}+\h}}
W^{(r)}_{\sigma,I} & \text{ for } r=00,01, \\[.1in]
\frac{z_{\sigma(a)}-z_{\sigma(a+1)}+\h}{z_{\sigma(a+1)}-z_{\sigma(a)}+\h}
W^{(r)}_{\sigma,I} & \text{ for } r=10,11.
\end{cases}
\]
\item  If $\sigma(a),\sigma(a+1)\not\in I$ then
\[
W^{(r)}_{\sigma s_{a,a+1},I}=
\begin{cases}
\phantom{\frac{z_{\sigma(a)}-z_{\sigma(a+1)}+\h}{z_{\sigma(a+1)}-z_{\sigma(a)}+\h}}
W^{(r)}_{\sigma,I} & \text{ for } r=00,10, \\[.1in]
\frac{z_{\sigma(a)}-z_{\sigma(a+1)}+\h}{z_{\sigma(a+1)}-z_{\sigma(a)}+\h}
W^{(r)}_{\sigma,I} & \text{ for } r=01,11.
\end{cases}
\]
\end{itemize}
\end{theorem}

\begin{proof}
These statements reduce algebraically to the special case $n=2,\sigma=\id, a=1$. That special case is equivalent with the four matrix product identities---obtained by concrete calculations--- in Section \ref{sec:weight}.
\end{proof}

\section{Existence of super stable envelopes}\label{sec:existence}

We prove the existence of $\Stab^{(r)}_{\sigma}$ maps by proving the existence of $\kappa^{(r)}_{\sigma,I}\in H^*_{\T}( \Grkn)$ elements satisfying axioms {\bf A0, A1, A2, A3} of Section \ref{sec:stab}.
Let us recall our description of $H^*_T(\Grkn)$
\begin{equation*}
\begin{tikzcd}
\C[t_1,\ldots,t_k,z_1,\ldots,z_n]^{S_k} \ar[r,"q",twoheadrightarrow] &
H^*_A(\Grkn) \ar[r,"\Loc",hook] &
\bigoplus_{I\in \I_k} \C[z_1,\ldots,z_n].
\end{tikzcd}
\end{equation*}
from \eqref{eq:Hdesc}.

\begin{lemma}
The tuple
\[
\left( W^{(r)}_{\sigma,J}(z_I,z,\h) \right)_{I\in \I_k} \in \bigoplus_{I\in \I_k} \C[z_1,\ldots,z_n]
\]
belongs to the image of $\Loc$.
\end{lemma}

\begin{proof}
First, the components of the tuple are indeed {\em polynomials}, according to Proposition \ref{prop:poly}.

We need to show that the tuple of polynomials satisfy the GKM condition. Let $U=K\cup \{i\}$, $V=K\cup \{j\}$ where $|K|=k-1$ and $i\not= j$. Substituting $z_i=z_j$ in
$W^{(r)}_{\sigma,J}(z_U,z,\h)$ and in $W^{(r)}_{\sigma,J}(z_V,z,\h)$ result identical functions. Hence
\[
\left.\left(W^{(r)}_{\sigma,J}(z_U,z,\h)- W^{(r)}_{\sigma,J}(z_V,z,\h)\right)\right|_{z_i=z_j}=0
\]
and, in turn, this implies that $z_i-z_j$ divides the difference $W^{(r)}_{\sigma,J}(z_U,z,\h)-W^{(r)}_{\sigma,J}(z_V,z,\h)$, what we wanted to prove.
\end{proof}

Therefore, the tuple $( W^{(r)}_{\sigma,J}(z_I,z,\h) )_{I\in \I_k}$ defines an element of $H^*_T(\Grkn)$, let us denote this element by $[ W^{(r)}_{\sigma,J}]$.

\begin{remark}
It is tempting to think that $W^{(r)}_{\sigma,J}$ represents $[ W^{(r)}_{\sigma,J} ]$ in the sense that $q(W^{(r)}_{\sigma,J})=[ W^{(r)}_{\sigma,J} ]$. This is, however, not correct because $W^{(r)}_{\sigma,J} \not \in \C[t_1,\ldots,t_k,z_1,\ldots,z_n]^{S_k}$. It is not a polynomial, it is a rational function (unless $r=10$). This rational function has the remarkable property that its $t=z_J$ substitutions are polynomials (satisfying the GKM condition), hence there is another function, this time a polynomial, whose $t=z_J$ substitutions are the same. That other polynomial would be the representative of $[ W^{(r)}_{\sigma,J} ]$ in $\C[t_1,\ldots,t_k,z_1,\ldots,z_n]^{S_k}$. For example for $n=4$ the element $\kappa^{\zo}_{\id,\{2\}}$ has the ``true'' polynomial representative given in \eqref{thatone}, but we named it with the rational function
\[
W^{\zo}_{\id,\{2\}}=\frac{\h\prod_{i=3}^4 (z_i-t_1) \prod_{1\leq i<j\leq 4} (z_j-z_i+\h)}{\prod_{i=2}^4(z_i-t_1+\h)}.
\]
The reader can verify that the polynomial in \eqref{thatone} and this rational function indeed have the same $t=z_J$ substitutions for all $J$.
\end{remark}

\begin{theorem} \label{thm:Stab=W}
The class
$
[
W^{(r)}_{\sigma,J}
]
$
satisfies the defining axioms for $\kappa^{(r)}_{\sigma,I}$.
\end{theorem}

\begin{proof}
The properties listed in Section \ref{sec:properties} verify the axioms {\bf A0-A3} required for $\kappa^{(r)}_{\sigma,I}$. Namely, the degree axiom {\bf A0} is implied by the fact that $W^{(r)}_{\sigma,J}$ has homogeneous degree $d^{(r)}$ (hence all its $t=z_I$ substitutions have that degree too). Axioms {\bf A1, A2, A3} are verified in Propositions~\ref{prop:principal},~\ref{prop:h-div},~\ref{prop:support}, respectively.
\end{proof}

\section{Geometrically defined local tensor coordinates on $(\C^2)^{\otimes n}$}

\subsection{Local tensor coordinates}\label{sec:ltc}

Let $v_1, v_2$ be the standard basis vectors of $\C^2$, and let us fix an element (``R-matrix'') $R(\zeta,\h) \in \End(\C^2 \otimes \C^2) \otimes \C(\zeta,\h)$. For $1\leq u,v \leq n$, $u\not= v$ let $R_{u,v}(\zeta)$ denote the element in $\End((\C^2)^{\otimes n})\otimes \C(\zeta)$ acting like $R(\zeta)$ in the $u$'th and $v$'th factor. Let us assume that the R-matrix satisfies the parametrized Yang-Baxter equation
\begin{equation}\label{eq:YB}
R_{12}(z_1-z_2) R_{13}(z_1-z_3) R_{23}(z_2-z_3) = R_{23}(z_2-z_3) R_{13}(z_1-z_3)R_{12}(z_1-z_2).
\end{equation}

\begin{definition}
Let $V$ be a vector space isomorphic with $(\C^2)^{\otimes n}$ (ie. of dimension $2^n$). A collection of linear isomorphisms
\[
\St_{\sigma} \in \Hom((\C^2)^{\otimes n},V)\otimes \C(z_1,\ldots,z_n,\h) \qquad\qquad \text{for }\sigma\in S_n
\]
is called {\em ``local tensor coordinates on $(\C^2)^{\otimes n}$''}, if for any $\sigma \in S_n$ and $a=1,\ldots,n-1$ we have
\[
\St^{-1}_{\sigma s_{a,a+1}} \circ \St_{\sigma}=R_{\sigma(a),\sigma(a+1)}(z_{\sigma(a+1)}-z_{\sigma(a)}).
\]
\end{definition}

For example the Yangain $\Y(\glv{2})$ action on $(\C^2)^{\otimes n}\otimes \C(z_1,\ldots,z_n,\h)$ induces local tensor coordinates on $(\C^2)^{\otimes n}$ with R-matrix
\begin{equation}\label{eq:Rproto}
\renewcommand\arraystretch{1.3}
R(\zeta)=
 \kbordermatrix{
& v_1\otimes v_1 & v_1\otimes v_2 & v_2\otimes v_1 & v_2\otimes v_2 \\
 v_1\otimes v_1 & \hskip .325 true cm 1 \hskip .325 true cm & 0 & 0 & 0 \\
 v_1\otimes v_2 & 0 & \frac{\zeta}{\h-\zeta} & \frac{\h}{\h-\zeta} & 0 \\
 v_2\otimes v_1 & 0 & \frac{\h}{\h-\zeta} & \frac{\zeta}{\h-\zeta} & 0 \\
 v_2\otimes v_2 & 0 & 0 & 0 & \hskip .325 true cm 1 \hskip .325 true cm
}
\in \End(\C^2 \otimes \C^2) \otimes \C(\zeta,\h).
\end{equation}
One of the achievements of \cite{MO} is a geometric construction of local tensor coordinates using the torus equivariant cohomology algebras of Nakajima quiver varieties. If the variety is $ \cup_k \rmTs\!\Grkn$ then the Maulik-Okounkov construction coincides with the $r=00$ case of what we will describe next.

\subsection{Super stable envelopes induce local tensor coordinates}\label{sec:StabLTC}
Let
\begin{equation}\label{eq:4R}
\renewcommand\arraystretch{1.3}
\begin{array}{ll}
R^{\zz}(\zeta)=
\left[\begin{array}{c|cc|c}
\hskip .23 true cm 1 \hskip .23 true cm & 0 & 0 & 0 \\ \hline
0 & \frac{\zeta}{\h-\zeta} & \frac{\h}{\h-\zeta} & 0 \\
0 & \frac{\h}{\h-\zeta} & \frac{\zeta}{\h-\zeta} & 0 \\\hline
0 & 0 & 0 & \hskip .23 true cm 1 \hskip .23 true cm
\end{array}\right],
&
R^{\oz}(\zeta)=
\left[\begin{array}{c|cc|c}
\hskip .23 true cm 1 \hskip .23 true cm & 0 & 0 & 0 \\ \hline
0 & \frac{\zeta}{\h-\zeta} & \frac{\h}{\h-\zeta} & 0 \\
0 & \frac{\h}{\h-\zeta} & \frac{\zeta}{\h-\zeta} & 0 \\\hline
0 & 0 & 0 & \frac{\h+\zeta}{\h-\zeta}
\end{array}\right],
\\
&
\\
R^{\zo}(\zeta)=
\left[\begin{array}{c|cc|c}
\frac{\h+\zeta}{\h-\zeta} & 0 & 0 & 0 \\ \hline
0 & \frac{\zeta}{\h-\zeta} & \frac{\h}{\h-\zeta} & 0 \\
0 & \frac{\h}{\h-\zeta} & \frac{\zeta}{\h-\zeta} & 0 \\\hline
0 & 0 & 0 & \hskip .23 true cm 1 \hskip .23 true cm
\end{array}\right],
&
R^{\oo}(\zeta)=
\left[\begin{array}{c|cc|c}
\frac{\h+\zeta}{\h-\zeta} & 0 & 0 & 0 \\ \hline
0 & \frac{\zeta}{\h-\zeta} & \frac{\h}{\h-\zeta} & 0 \\
0 & \frac{\h}{\h-\zeta} & \frac{\zeta}{\h-\zeta} & 0 \\\hline
0 & 0 & 0 & \frac{\h+\zeta}{\h-\zeta}
\end{array}\right].
\end{array}
\end{equation}
Each of these R-matrices satisfy the Yang-Baxter equation \eqref{eq:YB}.

\begin{remark} \label{rem:Rcheck}
In general, R-matrices come in two flavors, $R$ and $\check{R}$, whose relation is $\check{R}=P\circ R$ where $P$ is the operation permuting the factors of $\C^2\otimes \C^2$. Hence, the $\check{R}$ versions of the $R^{(r)}(\zeta)$ matrices above are obtained by replacing the middle $2\times 2$ submatrix
$\frac{1}{\h-\zeta} \begin{pmatrix}{\zeta} & \h \\ \h & \zeta \end{pmatrix}$
to
$\frac{1}{\h-\zeta} \begin{pmatrix} \h & {\zeta} \\ \zeta & \h  \end{pmatrix}$.
\end{remark}

In Section \ref{sec:Rmat} we will show the relationship between theses matrices and the Yangian R-matices of $\glsv{2}{0}, \glsv{1}{1}, \glsv{1}{1}, \glsv{0}{2}$ Lie superalgebras.

Now we give a geometric definition of local tensor coordinates on $(\C^2)^{\otimes n}$ with the R-matrices in \eqref{eq:4R}. For brevity, we will write $\C(z,\h)$ for $\C(z_1,\ldots,z_n,\h)$.

First, let us identify the vector spaces
\begin{equation}\label{eq23}
\begin{tikzcd}[row sep=0]
(\C^2)^{\otimes n} \otimes \C(z,\h)
\ar[r,equal] &
\bigoplus_{I\subset \{1,\ldots,n\}} \C(z,\h) \
\ar[r,equal] &
H_{\T}^*((\X_n^{(r)})^{\T})\otimes \C(z,\h) \\
\vin & \vin & \vin \\
v_{j_1}\otimes v_{j_2} \otimes \ldots \otimes v_{j_n} \ar[r,leftrightarrow]
& 1_I \ar[r,leftrightarrow] & 1\in H^*_{\T}(p_I)
\end{tikzcd}
\end{equation}
where $j_s=\begin{cases} 1 & \text{if } s\not\in I \\ 2 & \text{if } s\in I\end{cases}$. Recall from Section \ref{sec:Loc} that we defined $\Hb_n=H^*_{\T}(\X^{(r)}_n)\otimes \C(z,\h)$, and that the $\Loc$ map (see \eqref{eq:Loc1}) is an isomorphism from $\Hb_n$ to the vector space in \eqref{eq23}.

\begin{theorem}\label{thm:StabLTC}
Let $r\in \{00, 10, 01, 11\}$. The maps
\[
\Stab^{(r)}_{\sigma}: (\C^2)^{\otimes n} \otimes \C(z,\h)  \to \Hb_n \qquad\qquad (\sigma\in S_n)
\]
form local tensor coordinates on $(\C^2)^{\otimes n}$ with R-matrices given in \eqref{eq:4R}.
\end{theorem}

\begin{proof}
Define the {\em geometric R-matrix}
\[
\cR^{(r)}_{\sigma,\omega}=(\Stab^{(r)}_{\omega})^{-1} \Stab^{(r)}_{\sigma}.
\]
What we need to prove is that
\[
\cR^{(r)}_{\sigma s_{a,a+1},\sigma}=
R^{(r)}_{\sigma(a),\sigma(a+1)}(z_{\sigma(a+1)}-z_{\sigma(a)}).
\]

If ($\sigma(a)\in I$ and $\sigma(a+1)\not\in I$) or ($\sigma(a)\not\in I$ and $\sigma(a+1)\in I$) then we have
\[
\Stab^{(r)}_{\sigma}(1_I)=
[W^{(r)}_{\sigma,I}] =
\left[
\frac{z_{\sigma(a+1)}-z_{\sigma(a)}}{z_{\sigma(a+1)}-z_{\sigma(a)}+\h} W^{(r)}_{\sigma s_{a,a+1},I} +
\frac{\h}{z_{\sigma(a+1)}-z_{\sigma(a)}+\h} W^{(r)}_{\sigma s_{a,a+1},s_{\sigma(a),\sigma(a+1)}(I)}\right]
\]
according to Theorem \ref{thm:Stab=W} and Theorem \ref{thm:generalR} (in fact, writing $\sigma s_{a,a+1}$ for $\sigma$ in Theorem~\ref{thm:generalR}).
Hence
\begin{multline*}
\cR^{(r)}_{\sigma s_{a,a+1},\sigma}(1_I)=
\frac{z_{\sigma(a+1)}-z_{\sigma(a)}}{z_{\sigma(a+1)}-z_{\sigma(a)}+\h} 1_I +
\frac{\h}{z_{\sigma(a+1)}-z_{\sigma(a)}+\h} 1_{s_{\sigma(a),\sigma(a+1)}(I)}
\\ =
R^{(r)}_{\sigma(a),\sigma(a+1)}(z_{\sigma(a+1)}-z_{\sigma(a)})(1_I).
\end{multline*}

If $\sigma(a), \sigma(a+1)\in I$ then
\[
\Stab^{(r)}_{\sigma}(1_I)=
[W^{(r)}_{\sigma,I}] =
\begin{cases}
\phantom{\frac{z_{\sigma(a+1)}-z_{\sigma(a)}+\h}{z_{\sigma(a)}-z_{\sigma(a+1)}+\h}}
[W^{(r)}_{\sigma s_{a,a+1},I}] & \text{ for } r=00,01 \\
\frac{z_{\sigma(a+1)}-z_{\sigma(a)}+\h}{z_{\sigma(a)}-z_{\sigma(a+1)}+\h}
[W^{(r)}_{\sigma s_{a,a+1},I}]  & \text{ for } r=10,11,
\end{cases}
\]
according to Theorem \ref{thm:Stab=W} and Theorem \ref{thm:generalR} (in fact, writing $\sigma s_{a,a+1}$ for $\sigma$ in Theorem~\ref{thm:generalR}). Hence
\[
\cR^{(r)}_{\sigma s_{a,a+1},\sigma}(1_I)=
\begin{cases}
\phantom{\frac{z_{\sigma(a+1)}-z_{\sigma(a)}+\h}{z_{\sigma(a)}-z_{\sigma(a+1)}+\h}}
1_I & \text{ for } r=00,01 \\
\frac{z_{\sigma(a+1)}-z_{\sigma(a)}+\h}{z_{\sigma(a)}-z_{\sigma(a+1)}+\h}
1_I  & \text{ for } r=10,11,
\end{cases}
\]
which is equal to $R^{(r)}_{\sigma(a),\sigma(a+1)}(z_{\sigma(a+1)}-z_{\sigma(a)})(1_I)$.

The proof of the $\sigma(a), \sigma(a+1)\not\in I$ case is analogous.
\end{proof}

\subsection{Super weight functions induce local tensor coordinates}
Theorem \ref{thm:StabLTC} could have been proved without mentioning super weight functions, essentially only using the axioms of super stable envelopes. However, the proof we gave in Section \ref{sec:StabLTC} has the advantage that it actually proves another statement. Let $\W_{\sigma,n}\subset \C(t_1,\ldots,t_k,z_1,\ldots,z_n,\h)$ be the $\C(z_1,\ldots,z_n,\h)$-span of the functions $W_{\sigma,I}$ for $I\subset \{1,\ldots,n\}$. In fact, according to Theorem \ref{thm:generalR} this space is independent of $\sigma$, hence we will denote it by $\W_n$.

\begin{theorem}\label{thm:WLTC}
Let $r\in \{00,10,01,11\}$. The maps
\[
\begin{tabular}{cccccccc}
$W^{(r)}_{\sigma}:$ & $(\C^2)^{\otimes n} \otimes \C(z,\h)$  & $\to$ & $\W_n$ &&&& ($\sigma\in S_n$) \\
& $1_I$ & $\mapsto$ & $W^{(r)}_{\sigma,I}$ &&
\end{tabular}
\]
form local tensor coordinates on $(\C^2)^{\otimes n}$ with R-matrices given in \eqref{eq:4R}.
\end{theorem}

\begin{proof}
Our proof of Theorem \ref{thm:StabLTC} with the formal modification of writing $W^{(r)}_{\sigma}$ for $\Stab^{(r)}_{\sigma}$ and $W^{(r)}_{\sigma,I}$ for $[W^{(r)}_{\sigma,I}]$ proves this statement.
\end{proof}

The essence in this argument was that the R-matrix relations (Theorem \ref{thm:generalR}) hold not only for the cohomology classes $[W^{(r)}_{\sigma,I}]\in \Hb_n$ but for the rational functions $W^{(r)}_{\sigma,I} \in \W_n$ themselves: both skew arrows in the commutative diagram
\[
\begin{tikzcd}
& (\C^2)^{\otimes n} \otimes \C(z,\h) \arrow[dl,"W^{(r)}_\sigma"']  \arrow[dr,"\Stab^{(r)}_{\sigma}"] & \\
\W_n \arrow[rr,"{[\ ]}"] &  & \Hb_n.
\end{tikzcd}
\]
are local tensor coordinates.

That is, the super weight functions are {\em those} representatives 
of super stable envelopes  that respect the R-matrix property. Such a choice of representative for an interesting cohomology class follows the tradition started with Schubert polynomials \cite{LS}. Schubert polynomials represent fundamental classes of Schubert varieties, and these representatives are chosen in a way to be consistent with the Lascoux-Sch\"utzenberger recursion of Schubert classes.

\section{Yangian R-matrices of $\glsv{2}{0}, \glsv{1}{1}, \glsv{0}{2}$}\label{sec:Rmat}

Consider the four splittings
\[
\C\langle v_1\rangle_{\even} \oplus \C\langle v_2\rangle_{\even}, \quad
\C\langle v_1\rangle_{\even} \oplus \C\langle v_2\rangle_{\odd}, \quad
\C\langle v_1\rangle_{\odd} \oplus \C\langle v_2\rangle_{\even}, \quad
\C\langle v_1\rangle_{\odd} \oplus \C\langle v_2\rangle_{\odd}, \quad
\]
of $\C^2\langle v_1,v_2\rangle$. The corresponding Lie superalgebras are $\glsv{2}{0}, \glsv{1}{1},\glsv{1}{1}, \glsv{0}{2}$, where the middle two are of course isomorphic, and $\glsv{2}{0},\glsv{0}{2}$ are both isomorphic with the ordinary $\glv{2}$.  The Yangian R-matrices for these Lie superalgebras are all
\[
1+uP \in \End(\C^2 \otimes \C^2) \otimes \C(u)
\]
with the key difference that the ``permutation of factors'' operator, $P$, is meant with the usual convention: when odd vectors are permuted a $(-1)$-sign is introduced
\cite{Gow, Zhang1, Zhang2}.
Namely, the four Yangian R-matrices are (in the ordered basis $v_1\otimes v_1, v_1\otimes v_2, v_2\otimes v_1, v_2\otimes v_2$)
\begin{equation}\label{eq:YangianR}
\renewcommand\arraystretch{1.3}
\begin{array}{ll}
R_{00}(u)=
\left[\begin{array}{c|cc|c}
1+u  & 0 & 0 & 0 \\ \hline
0 & 1 & u & 0 \\
0 & u & 1 & 0 \\\hline
0 & 0 & 0 & 1+u
\end{array}\right],
&
R_{10}(u)=
\left[\begin{array}{c|cc|c}
1+u  & 0 & 0 & 0 \\ \hline
0 & 1 & \hskip .18 true cm u \hskip .18 true cm & 0 \\
0 & \hskip .18 true cm u \hskip .18 true cm & 1 & 0 \\\hline
0 & 0 & 0 & 1-u
\end{array}\right],
\\ & \\
R_{01}(u)=
\left[\begin{array}{c|cc|c}
1-u & 0 & 0 & 0 \\ \hline
0 & 1 & u & 0 \\
0 & u & 1 & 0 \\\hline
0 & 0 & 0 & 1+u
\end{array}\right], &
R_{11}(u)=
\left[\begin{array}{c|cc|c}
1-u  & 0 & 0 & 0 \\ \hline
0 & 1 & -u & 0 \\
0 & -u & 1 & 0 \\\hline
0 & 0 & 0 & 1-u
\end{array}\right].
\end{array}
\end{equation}
The $\check{R}$ version of R-matrices are obtained as $\check{R}=P\circ R$ (with the respective $P$ operator), hence they are
\begin{equation}\label{eq:YangianRc}
\renewcommand\arraystretch{1.3}
\begin{array}{ll}
\check{R}_{00}(u)=
\left[\begin{array}{c|cc|c}
\hskip .13 true cm 1+u  \hskip .13 true cm & 0 & 0 & 0 \\ \hline
0 & u & 1 & 0 \\
0 & 1 & u & 0 \\\hline
0 & 0 & 0 & 1+u
\end{array}\right],
&
\check{R}_{10}(u)=
\left[\begin{array}{c|cc|c}
\hskip .13 true cm 1+u  \hskip .13 true cm   & 0 & 0 & 0 \\ \hline
0 & u & \hskip .18 true cm 1 \hskip .18 true cm & 0 \\
0 & \hskip .18 true cm 1 \hskip .18 true cm & u & 0 \\\hline
0 & 0 & 0 & -1+u
\end{array}\right],
\\ & \\
\check{R}_{01}(u)=
\left[\begin{array}{c|cc|c}
-1+u & 0 & 0 & 0 \\ \hline
0 & u & 1 & 0 \\
0 & 1 & u & 0 \\\hline
0 & 0 & 0 & 1+u
\end{array}\right], &
\check{R}_{11}(u)=
\left[\begin{array}{c|cc|c}
-1+u  & 0 & 0 & 0 \\ \hline
0 & u & -1 & 0 \\
0 & -1 & u & 0 \\\hline
0 & 0 & 0 & -1+u
\end{array}\right].
\end{array}
\end{equation}
If we divide these $\check{R}$-matrices by $1+u$, and then substitute $u=-\h/\zeta$, then we obtain exactly the $\check{R}$ matrices of Remark \ref{rem:Rcheck}. Therefore, the $\check{R}$-matrices obtained from the geometry (namely, the super stable envelopes) of $\X_{n}^{(r)}$ spaces are the Yangian $\check{R}$-matrices of $\glsv{2}{0}$, $\glsv{1}{1}$, $\glsv{1}{1}$, $\glsv{0}{2}$.

\end{document}